\documentclass[a4paper,10pt]{amsart}
\title{Ideal Topologies in Higher Descriptive Set Theory}
\author[P.\ Holy]{Peter Holy}
\address{University of Udine\\
Via delle Scienze 206\\
33100 Udine\\
Italy}
\email{peter.holy@uniud.it}
\author[M.\ Koelbing]{Marlene Koelbing}
\address{Universit\"at Wien\\
Institut f\"ur Mathematik\\
Kurt G\"odel Research Center\\
Kolingasse 14--16\\
1090 Wien\\
Austria}
\email{marlenekoelbing@web.de}

\author[P.\ Schlicht]{Philipp Schlicht}
\address{Universit\"at Bonn, Mathematisches Institut, Endenicher Allee 60, 53115 Bonn, Germany,
and
University of Bristol, School of Mathematics, Fry Building, Woodland Road, Bristol, BS8 1UG, UK}
\email{schlicht@math.uni-bonn.de}

\author[W.\ Wohofsky]{Wolfgang Wohofsky}
\address{Universit\"at Wien\\
Institut f\"ur Mathematik\\
Kurt G\"odel Research Center\\
Kolingasse 14--16\\
1090 Wien\\
Austria}
\email{wolfgang.wohofsky@gmx.at}

\makeatletter
\@namedef{subjclassname@2020}{\textup{2020} Mathematics Subject Classification}
\makeatother 

\usepackage{amsmath,amssymb,amscd,amsthm,amsfonts,amstext,amsbsy,mathrsfs,hyperref,upgreek,mathtools,stmaryrd,enumitem}

\usepackage{fontenc}
\usepackage{amsmath}
\usepackage{amsfonts}
\usepackage{amssymb}
\usepackage{amsthm}
\usepackage{graphicx}
\usepackage[english]{babel}
\usepackage{array}

\usepackage[textsize=footnotesize,color=yellow!70, bordercolor=white,disable]{todonotes}
\usepackage{hyperref}

\newtheorem{definition2}{Definition}[section]

\newtheorem{lemma}[definition2]{Lemma}
\newtheorem{remark2}[definition2]{Remark}
\newtheorem{theorem}[definition2]{Theorem}
\newtheorem{observation}[definition2]{Observation}
\newtheorem*{Mobservation2}{Main Observation}
\newtheorem{example2}[definition2]{Example}
\newtheorem{corollary}[definition2]{Corollary}
\newtheorem{proposition}[definition2]{Proposition}

\newtheorem{question}{Question}

\newtheorem*{claim*}{Claim}

\newenvironment{definition}{\begin{definition2} \upshape}{\end{definition2}}
\newenvironment{remark}{\begin{remark2} \upshape}{\end{remark2}}

\newenvironment{example}{\begin{example2} \upshape}{\end{example2}}

\newenvironment{enumerate-(a)}{\begin{enumerate}[label={\upshape (\alph*)}, leftmargin=2pc]}{\end{enumerate}}
\newenvironment{enumerate-(a)-r}{\begin{enumerate}[label={\upshape (\alph*)}, leftmargin=2pc,resume]}{\end{enumerate}}
\newenvironment{enumerate-(A)}{\begin{enumerate}[label={\upshape (\Alph*)}, leftmargin=2pc]}{\end{enumerate}}
\newenvironment{enumerate-(A)-r}{\begin{enumerate}[label={\upshape (\Alph*)}, leftmargin=2pc,resume]}{\end{enumerate}}
\newenvironment{enumerate-(i)}{\begin{enumerate}[label={\upshape (\roman*)}, leftmargin=2pc]}{\end{enumerate}}
\newenvironment{enumerate-(i)-r}{\begin{enumerate}[label={\upshape (\roman*)}, leftmargin=2pc,resume]}{\end{enumerate}}
\newenvironment{enumerate-(I)}{\begin{enumerate}[label={\upshape (\Roman*)}, leftmargin=2pc]}{\end{enumerate}}
\newenvironment{enumerate-(I)-r}{\begin{enumerate}[label={\upshape (\Roman*)}, leftmargin=2pc,resume]}{\end{enumerate}}
\newenvironment{enumerate-(1)}{\begin{enumerate}[label={\upshape (\arabic*)}, leftmargin=2pc]}{\end{enumerate}}
\newenvironment{enumerate-(1)-r}{\begin{enumerate}[label={\upshape (\arabic*)}, leftmargin=2pc,resume]}{\end{enumerate}}

\newcommand{\Lim}{Lim}

\newcommand{\id}{id}

\newcommand{\PP}{\mathbb P}

\newcommand{\NS}{\mathsf{NS}}
\newcommand{\Closed}{\mathrm{Closed}}
\newcommand{\Club}{\mathrm{Club}}

\newcommand{\I}{\mathsf I}
\newcommand{\cof}{\mathsf{cof}}

\newcommand{\conc}{\smallfrown}
\newcommand{\club}{\mathrm{Club}}
\newcommand{\cluf}{\mathcal{C}}

\newcommand{\dom}{\mathrm{dom}}

\newcommand{\non}{\mathsf{non}}
\newcommand{\restr}{\upharpoonright}
\newcommand{\Restr}{\upharpoonright\!\upharpoonright}

\newcommand{\zero}{\mathbf 0}
\newcommand{\one}{\mathbf 1}

\newcommand{\II}{\mathcal{I}}
\newcommand{\JJ}{\mathcal{J}}

\newcommand{\with}{\mid}

\newcommand{\reap}{\mathfrak{r}(\kappa)}

\newcommand{\edinburgh}{\text{Edinburgh}}

\DeclareMathSymbol{\shortminus}{\mathbin}{AMSa}{"39}
\newcommand{\prevspl}[2]{{#1}^{\,\shortminus}_{#2}}
\newcommand{\marlene}{p}

\newcommand{\todoi}[1]{}
\newcommand{\todoe}[1]{}


\thanks{The first author would like to thank Vincenzo Dimonte for some helpful discussions related to the topics of this paper. The authors thank the anonymous referee for some useful comments. The research of the first author was supported by the Italian PRIN 2017 Grant \emph{Mathematical Logic: models, sets, computability}.
The second author was supported by a \"OAW Doc Stipendium and by the FWF grants number I1921, P28420 and Y1012.
This project has received funding from the European Union's Horizon 2020 research and innovation programme under the Marie Sk{\l}odowska-Curie grant agreement No 794020 of the third author (Project \emph{IMIC: Inner models and infinite computations}).
The third author was partially supported by FWF grant number I4039 and EPSRC grant number EP/V009001/1.
The fourth author was supported by German Research Foundation (DFG) under grant SP 683/4-1, and FWF grants number Y1012 and I4039}

\newcommand{\bd}{\mathsf{bd}}
\newcommand{\ub}{\mathsf{ub}}

\newcommand{\heart}{\mathsf{Fn}}

\newcommand{\Jsubseq }{\hookrightarrow_{\JJ}}

\newcommand{\X}{A}
\newcommand{\B}{B}

\begin{document}

\begin{abstract}
  We investigate generalizations of the topology of the higher Cantor space on $2^\kappa$, based on arbitrary ideals rather than the bounded ideal on $\kappa$. Our main focus
  is on the topology induced by the nonstationary ideal, and we
  call this topology the \emph{nonstationary topology}, or also the \emph{Edinburgh topology} on~$2^\kappa$.

 It may be of independent interest that as a side result, we show $\kappa$-Silver forcing to satisfy a strong form of Axiom $A$ not only if $\kappa$ is inaccessible (which is well-known), but also under the assumption $\diamondsuit_\kappa$.
\end{abstract}

\maketitle

\section{Introduction}

Let $\kappa$ be a regular and uncountable cardinal, and let $\bd_\kappa$ denote the bounded ideal on $\kappa$. 
For an ideal~$\II$ on~$\kappa$, let $\II^+=\{ x \subseteq \kappa \mid x\notin\II\}$
denote the 
collection of 
$\II$-positive sets. 
Let $\ub_\kappa$ denote the collection of all unbounded subsets of~$\kappa$, i.e.,  $\ub_\kappa=(\bd_\kappa)^+$. 
Let $\II^*=\{\kappa\setminus x\mid x\in\II\}$ denote the filter dual to $\II$. Let $\NS_\kappa$ denote the nonstationary ideal on $\kappa$. We want to consider different topologies on~$2^\kappa$, induced by ideals other than the bounded ideal on $\kappa$. Let $\II$ be a ${<}\kappa$-complete proper ideal on $\kappa$ that extends $\bd_\kappa$ throughout our paper.

\begin{definition}
  For any set $I$, let $\heart_{I}=\{f\mid f\colon D\to 2$ is a function with $D\in I\}$. For an ideal $\II$ as above, the \emph{$\II$-topology} $\tau_\II$ of $\II$-open sets is provided by the basis $\{[f]\mid f\in\heart_\II\}$ of $\II$-clopen subsets of $2^\kappa$ (we also call those $\II$-clopen sets \emph{$\II$-cones}), where, for any partial function $f\colon\kappa\to 2$, $[f]=\{g\in 2^\kappa\mid f\subseteq g\}$, and where we say that a set is \emph{$\II$-clopen} / \ldots\ if it is clopen / \ldots\ in the $\II$-topology. An \emph{ideal topology} is an $\II$-topology for an ideal $\II$ as above. If $\II=\NS_\kappa$, we also say that a set is \emph{Edinburgh open} / \ldots\ in case it is $\II$-open / \ldots,
  and we refer to the $\II$-topology on $2^\kappa$ as the \emph{Edinburgh topology}
or the \emph{nonstationary topology}
on $2^\kappa$.
As usual, we say that a set is \emph{open} / \ldots\ in case it is $\bd_\kappa$-open / \ldots,
    and we refer to the $\bd_\kappa$-topology on $2^\kappa$ as the \emph{bounded topology} on $2^\kappa$.
\end{definition}

The topology of the higher Cantor space on $2^\kappa$ is the $\bd_\kappa$-topology.
Clearly, if $\II_1\supseteq\II_0$, then $\tau_{\II_1}$ refines $\tau_{\II_0}$.
In particular, the Edinburgh topology on $2^\kappa$ thus refines the topology of the higher Cantor space $2^\kappa$.
Since the $\bd_\kappa$-topology on $2^\kappa$ is Hausdorff, the same holds for every $\II$-topology.
 A simple observation shows that as soon as we allow our ideals to contain unbounded sets, they yield topologies with the maximal possible number of open sets:

\begin{observation}\label{observation:manyopensets}
  If $\II$ contains an unbounded subset of $\kappa$, then:
  \begin{enumerate}
    \item There are $2^\kappa$-many disjoint basic $\II$-open sets whose union is $2^\kappa$.
    \item $|\tau_\II|=2^{2^\kappa}$.
  \end{enumerate}
\end{observation}
\begin{proof}
  \begin{enumerate}
    \item Let $A$ be an unbounded subset of $\kappa$ in $\II$, and let $F=\{f \mid f\colon A\to 2\}$. Then, the $2^\kappa$-many $\II$-cones $[f]$ for $f\in F$ are pairwise disjoint, and their union is all of $2^\kappa$, as desired.
    \item For any $X\subseteq F$, let $\mathcal O_X=\bigcup_{f\in X}[f]$. Then, $X\neq Y$ implies $\mathcal O_X\neq\mathcal O_Y$.\qedhere
  \end{enumerate}
\end{proof}
Note that (1) implies in particular a strong failure of (generalized) compactness for $\II$-topologies whenever $\II$ contains an unbounded subset of $\kappa$.
Also note that (1) implies that, in such cases, the $\II$-topology does not have a basis of size strictly less than $2^\kappa$;
this is a strong failure of ``generalized second countability''.
On the other hand,
$\{[f]\mid f\in\heart_\II\}$ is a canonical basis of size~$2^\kappa$, hence
the $\II$-topology has
weight~$2^\kappa$.

\medskip

One of the most basic topological results holds outright for arbitrary ideal topologies, by the usual argument, which we would nevertheless like to present for the convenience of our readers.

\begin{proposition}\label{the:BCT} (Baire category theorem for ideal topologies)
The intersection of $\kappa$-many $\II$-open dense sets is $\II$-dense.\footnote{We write $\II$-open dense to mean $\II$-open and $\II$-dense. Furthermore, using our above convention, a set being $\II$-dense means that it intersects every $\II$-open set.}
\end{proposition}

Note that this statement is equivalent to the fact that for every $f\in \heart_\II$, the
 $\II$-cone $[f]$ is not $\II$-meager.
\begin{proof}[Proof of Proposition~\ref{the:BCT}]
 Let $(D_\alpha)_{\alpha<\kappa}$ be a sequence of $\II$-open dense sets. For every $\alpha<\kappa$, there exists a set $I_\alpha$ and a sequence $\langle f^\alpha_i\mid i\in I_\alpha\rangle$, with each $f^\alpha_i\in\heart_\II$, such that $D_\alpha=\bigcup_{i\in I_\alpha}[f^\alpha_i]$. Let $f_0\in \heart_\II$. We construct a $\subseteq$-increasing sequence of functions $\langle f_\alpha\mid\alpha<\kappa\rangle$ in $\heart_\II$ such that $f_\kappa:=\bigcup_{\alpha<\kappa}f_\alpha$ satisfies $[f_\kappa]\subseteq\bigcap_{\alpha<\kappa}D_\alpha$.
  Given an ordinal $\alpha<\kappa$ and $f_\alpha\in\heart_\II$, since $D_\alpha$ is $\II$-dense, $D_\alpha\cap [f_\alpha]\neq \emptyset$. Therefore, there exists $i\in I_\alpha$ such that $[f^\alpha_i]\cap [f_\alpha]\neq \emptyset$, yielding that $f_{\alpha+1}:=f^\alpha_i\cup f_\alpha$ is a function. Since $\II$ is an ideal, $f_{\alpha+1}\in\heart_\II$.
  Assume now that $\alpha<\kappa$ is a limit ordinal, and that $\langle f_\beta\mid\beta<\alpha\rangle$ has been constructed. Then, $\bigcup_{\beta<\alpha}f_\beta$ is a function, and, using that $\II$ is ${<}\kappa$-complete, $\bigcup_{\beta<\alpha}f_\beta\in\heart_\II$.
 In the end, $f_\kappa$ is clearly as desired.
\end{proof}

The above proof essentially also shows the following
slightly stronger result.

\begin{corollary}\label{corollary:mycielski} (Mycielski's theorem for ideal topologies)
  The intersection of $\kappa$-many $\II$-open dense sets contains a perfect set (in the sense of the bounded topology on $2^\kappa$), that is, a closed set that is homeomorphic to the higher Cantor space $2^\kappa$.
\end{corollary}
\begin{proof}
  As for Proposition~\ref{the:BCT}, but extending $f_0$ to two incompatible functions, and then extending those to some $f_1$ and $f'_1$ in the same way that we extended $f_0$ to $f_1$ in the proof of Proposition~\ref{the:BCT}. Continuing now with these extensions, and carrying on like this throughout all $\kappa$-many stages of our construction, we eventually obtain our desired perfect set in the intersection of our $\kappa$-many $\II$-open dense sets.
\end{proof}

We will later show that in many situations, the nonstationary topology satisfies an even stronger form of Mycielski's theorem (see Theorem~\ref{theorem:cones2}).

Observe that whenever $\mathcal{B} \subseteq \II$ is a basis for~$\II$ (i.e., for every $A \in \II$ there exists a $B \in \mathcal{B}$ with $A \subseteq B$), then each point
in~$2^\kappa$ has a neighbourhood basis of the same size as~$\mathcal{B}$; conversely, each neighbourhood basis yields a basis for~$\II$ of the same size.
In particular, $2^\kappa$ with the $\II$-topology is ``generalized first countable'' (i.e., each point has a neighbourhood basis of size~$\kappa$) if and only if $\II$ has a basis of size~$\kappa$. Note that
$\{ \alpha \with \alpha < \kappa \} \subseteq \bd_\kappa$ is a basis (of size~$\kappa$) for the bounded ideal, hence $2^\kappa$ with the bounded topology is always ``generalized first countable''.
There are also ideals
other than $\bd_\kappa$ which have a basis of size~$\kappa$ (e.g., the ideal from
Example~\ref{example:nonhomogeneous}).

We say that a set is $\II$-$F_\kappa$ if it is
the union of $\kappa$-many $\II$-closed sets.
The following proposition generalizes the fact that, in the bounded topology, every open set is $F_\kappa$ (even, if $\kappa^{<\kappa} > \kappa$; see~\cite[Lemma~4.15]{Luca}).

\begin{proposition}
If $\II$ has a basis of size $\kappa$, then every $\II$-open set is $\II$-$F_\kappa$.
\end{proposition}

\begin{proof} Let $\mathcal{B}=\{ B_i \mid i<\kappa\}\subseteq \II$ be a basis for $\II$, and let $X\subseteq 2^\kappa$ be $\II$-open. Every $x\in X$ is in the $\II$-interior of $X$, so there exists $A_x\in \mathcal{B}$ such that $[x\restr A_x] \subseteq X$. Clearly, $X=\bigcup_{x\in X} [x\restr A_x]$. Let
$$
X_i:=\bigcup\{ [x\restr A_x] \mid x\in X \textrm{ and } A_x=B_i\}.
$$
 It is easy to see that the sets $X_i$ are $\II$-closed (in fact $\II$-clopen), so clearly $X=\bigcup_{i<\kappa} X_i$ is $\II$-$F_\kappa$.
\end{proof}

On the other hand, there are ideals for which the usual implication
``every $\II$-open set is $\II$-$F_\kappa$'' fails (e.g., for the nonstationary ideal~$\NS_\kappa$; see
Corollary~\ref{corollary:tall} and
Theorem~\ref{theorem:ubnotfsigma}).

 For the nonstationary topology, we also have the property that cones are isomorphic to the whole space (as is the case for the standard bounded topology).

\begin{lemma}\label{lem:every_cone_homeomorphic}
  If $\II=\NS_\kappa$, then
  the space $2^\kappa$ with the $\II$-topology
  is homeomorphic to
  any $\II$-cone with the induced topology.

More precisely,
if $f\in \heart_{\II}$
 then there exists
$\rho\colon [f] \rightarrow 2^\kappa$
which is a homeomorphism both with respect to the bounded
topology and with respect to the
nonstationary topology
(taking the respective induced topologies on~$[f]$).

\end{lemma}
\begin{proof}
  Let $f\in\heart_\II$ be given, so that $A:=\dom(f)$ is nonstationary, and let us consider the $\II$-cone $[f]$. We want to construct a homeomorphism between $[f]$ and $2^\kappa$ based on a bijection $\pi$ between $B:=\kappa\setminus A$ and $\kappa$. Let $C\subseteq B$ be a club subset of $\kappa$, such that, in order to simplify the argument to follow, $B\setminus C$ is an unbounded subset of $\kappa$. Let $\langle c_\alpha\mid\alpha<\kappa\rangle$ be the increasing enumeration of $C$, and let $\langle b_\alpha\mid\alpha<\kappa\rangle$ be the increasing enumeration of $B\setminus C$. We define $\pi\colon B\to\kappa$ by setting $\pi(c_\alpha)=2\cdot\alpha$, and letting $\pi(b_\alpha)=2\cdot\alpha+1$.

  The point of our construction is now that $\pi$ is a bijection between $B$ and $\kappa$ such that if $D\subseteq B$, then $D$ is stationary if and only if $\pi[D]$ is stationary: This follows because if $E\subseteq B$ contains a club subset of~$\kappa$, then also $E\cap C$ contains a club subset of~$\kappa$, and since $\pi$ is continuous on $C$, it follows that $\pi[E]$ contains a club subset of~$\kappa$. On the other hand, if $E\subseteq\kappa$ contains a club subset of~$\kappa$, then its restriction to the even ordinals contains a club subset of~$\kappa$, and its pointwise preimage under $\pi$ contains a club subset of $C\subseteq B$.

  Next, we use $\pi$ to induce a bijection $\rho\colon[f]\to 2^\kappa$ in a natural way:
For each $x \in [f]$, we simply define $\rho(x)$ by letting
$\rho(x)(\alpha) := x(\pi^{-1}(\alpha))$
for each $\alpha \in \kappa$.
It is easy to see that $\rho$ is a homeomorphism between $2^\kappa$ with the bounded topology and $[f]$ (with the induced topology).
We will finish our argument by showing that $\rho$ is also a homeomorphism
between~$2^\kappa$ with the nonstationary topology and $[f]$ (with the induced topology).

  It suffices to show that both $\rho$ and its inverse map preserve basic open sets. A basic open subset of $[f]$ in its induced topology is of the form $[g]$ for $g\supseteq f$ in $\heart_\II$. Then, $\rho[[g]]=[h]$, where
    $\dom(h)=\pi[\dom(g) \setminus A]$ and $h(\pi(\alpha))= g(\alpha)$ for every $\alpha\in\dom(g)\setminus A$.
Since $\dom(h)$ is nonstationary by our above arguments, this shows that $[h]$ is a basic open set in the nonstationary topology.

  For the other direction, assume that $[h]$ is a basic open subset of $2^\kappa$ in the nonstationary topology. Then, $\rho^{-1}[[h]]=[g]$ where $g\supseteq f$ is such that $\dom(g)=A\cup\pi^{-1}[\dom(h)]$ and for $\alpha\in\dom(g)\setminus A$, $g(\alpha)=h(\pi(\alpha))$. Again by our above arguments, $\dom(g)$ is nonstationary, hence $[g]$ is a basic open set in the induced topology on $[f]$, as desired.
\end{proof}

There are counterexamples to the above homogeneity property for other $\II$-topologies:
\begin{example}\label{example:nonhomogeneous}
 Assume $2^{<\kappa}=\kappa$. Let
 $A$
 be an unbounded subset of $\kappa$ which also has an unbounded complement, and let $\II$ be the ideal generated by $\bd_\kappa$ together with the set $A$ -- that is, $B\in\II$ if and only if $B\setminus A$ is a bounded subset of $\kappa$. Then, for any $f\colon A\to 2$, $2^\kappa$ with the $\II$-topology is not homeomorphic to $[f]$ with its induced topology.
\end{example}
\begin{proof}
  Homeomorphic topological spaces need to have the same number of open sets.
  However,
  by Observation \ref{observation:manyopensets}, $|\tau_\II|=2^{2^\kappa}$, while there are only $2^\kappa$-many open sets in the induced topology on $[f]$, for it is clearly homeomorphic to the bounded topology on $2^\kappa$ (which has a basis of size~$2^{<\kappa} = \kappa$).
\end{proof}

Let us show that every $\II$-topology is \emph{homogeneous}, in the sense that for any two elements $x,y\in 2^\kappa$, there is a homeomorphism of $2^\kappa$ with respect to $\tau_\II$ that maps $x$ to $y$. Let $\zero$ denote the function with domain $\kappa$ and constant value $0$, and let $\one$ denote the function with domain $\kappa$ and constant value $1$.

\begin{proposition}\label{prop:flip}
The $\II$-topology is homogeneous for any ideal $\II$.
\end{proposition}
\begin{proof}
  It suffices to provide, for each $x\in 2^\kappa$, a homeomorphism $H\colon 2^\kappa\rightarrow 2^\kappa$ with $H(\zero)=x$. For any subset $x$ of $\kappa$, we
  define\footnotemark{} the function $H_x\colon 2^\kappa\rightarrow 2^\kappa$ by $H_x(y)(i)=1-y(i)$ for all $i\in x$ and $H_x(y)(i)=y(i)$ otherwise. Clearly, $H_x$ is a bijection, and $H_x$ is a homeomorphism, since $H_x[[f]]=[g]$
  whenever $f,g\in\heart_\II$ such that $\dom(g) = \dom(f)$, and
  $g(i)=1-f(i)$ for $i\in \dom(f)\cap x$ and $g(i)=f(i)$ for $i\in \dom(f)\setminus x$. Clearly, $H_x(\zero)=x$.
\end{proof}

\footnotetext{Note that $H_x(y) = x + y$ for each $y \in 2^\kappa$, where $x + y$ is the bitwise sum (modulo $2$) of $x$ and~$y$.
In fact, it is easy to check that $2^\kappa$ together with the operation $+$ is a topological group (with respect to the $\II$-topology).}

The following trivial observation will be useful later on:
\begin{observation}\label{observation:homeomorphiccones}
Let $s\in 2^{<\kappa}$. Then $2^\kappa$ is homeomorphic to $[s]$ with respect to the $\II$-topology.

In particular, there are $\kappa$-many disjoint $\II$-cones that are homeomorphic to $2^\kappa$ with the $\II$-topology.
\end{observation}
\begin{proof}
We show that
$2^\kappa$ is homeomorphic to $[s]$ with respect to the $\II$-topology, using the bijection $\pi\colon 2^\kappa\to[s]$ which maps $x$ to $s^\smallfrown x$ (where we are thinking of $s$ and of $x$ as sequences of $0$'s and $1$'s, and $s^\smallfrown x\in 2^\kappa$ denotes their concatenation). To see that $\pi$ is an $\II$-homeomorphism, it is enough to show that
\begin{equation}\label{eq:arithmetic}
A\in \II \text{ if and only if } \{|s| + \beta \mid \beta \in A\}\in \II.
\end{equation}
 Observe that there exists $\gamma<\kappa$ such that $|s|+\beta=\beta$ for every $\beta\geq \gamma$, hence $\gamma\cup A = \gamma \cup \{|s| + \beta \mid \beta \in A\}$. Since $\bd_\kappa\subseteq \II$,
it follows that~\eqref{eq:arithmetic} holds.

The second statement of the observation easily follows by picking $\kappa$-many incompatible functions in $2^{<\kappa}$.
\end{proof}

\section{An overview}

Before we embark on the main parts of the paper, we would like to give a quick overview of the contents of this paper, and of some of our main results.

\medskip

Section \ref{section Borel hierarchy} studies
the class of Borel sets of ideal topologies.
We first investigate the complexity of some canonical subsets of $2^\kappa$ in $\II$-topologies.
Theorem~\ref{theorem:ubnotfsigma} shows that the bounded ideal is never $\II$-$G_\kappa$. Since the bounded ideal is $\II$-closed in many cases (and in particular when $\II=\NS_\kappa$, see Corollary \ref{corollary:tall}), this also shows that $\II$-closed sets need not be $\II$-$G_\kappa$ in general.
The set $\Club_\kappa$ of club subsets of $\kappa$ is always $\II$-$G_\kappa$;
Theorem \ref{theorem:clubnotfsigma} shows that
$\Club_\kappa$
is not $\II$-$F_\kappa$ when $\II$ is the nonstationary ideal on $\kappa$, and we also observe that in this case,
$\Club_\kappa$ can neither be $\II$-open nor $\II$-closed.
However, Observation \ref{observation:stationaryclubclosed} and Proposition \ref{proposition:clubopen} show that for certain choices of~$\II$ other than $\NS_\kappa$, $\Club_\kappa$ can sometimes be $\II$-closed or $\II$-open.
Corollary~\ref{club filter is no Edinburgh Borel} shows that the club filter is not $\II$-Borel in the nonstationary topology.
One of the main questions left open by the results of Section \ref{section Borel hierarchy} is certainly whether
there is an $\II$-Borel hierarchy which resembles
the usual Borel hierarchy on the Cantor space.

\medskip

\textbf{Question} (see Question \ref{question Borel hierarchy}).
Assuming  $\II=\NS_\kappa$, do the $\II$-Borel sets form a strict hierarchy of length $\kappa^+$?

\medskip

In Section \ref{section:sequences}, we investigate possible notions of convergence, accumulation points and subsequences in ideal topologies.

In Section \ref{section:connection}, we make some remarks on the connection between ideal topologies and forcing topologies, with the former being a special case of the latter. In particular, we show that the topology induced by $\kappa$-Silver forcing is exactly the nonstationary topology on $2^\kappa$.

In Section \ref{section:axioma}, we investigate a strengthening of Axiom~$A$ that was introduced as Axiom~$A^*$ in \cite{fkk}, and show
in Theorem \ref{theorem:axioma}
that if $\diamondsuit_\kappa$ holds, then this axiom is satisfied by $\kappa$-Silver forcing.
This was previously known under the assumption that $\kappa$ is inaccessible. We say that
a regular uncountable
cardinal~$\kappa$ is \emph{simple} in case it is inaccessible or $\diamondsuit_\kappa$ holds.

In Section \ref{section:premeager}, we start an investigation of the connections between meager and $\II$-meager sets. We make use of the results of Section \ref{section:axioma} in Theorem \ref{theorem:cones2}, where we show that for simple cardinals $\kappa$, the notions of meager and of nowhere dense sets coincide for the nonstationary topology on $\kappa$. We do not know whether this holds without the assumption of simplicity (see Question \ref{question:meagernowheredense})
or for
other ideal topologies
(see Question~\ref{question:other_ideal_Axiom_A}). In Theorem \ref{theorem:cones}, again for simple cardinals $\kappa$, we show that every comeager subset of $2^\kappa$ contains a dense set that is open in the nonstationary topology. From this, in Corollary \ref{corollary:imeagerbairemeager} we infer that if $\kappa$ is simple and $\II\supseteq\NS_\kappa$, then all $\II$-meager sets with the Baire property are in fact meager.

In Section \ref{section:reaping}, we investigate the reaping number $\reap$ at $\kappa$ and some of its variants, which we show to agree with each other in case $\kappa$ is simple
(see Theorem~\ref{theorem:comeagerreaping}).

In Section \ref{section:meager}, we continue our investigation of the notion of meagerness in ideal topologies. In Proposition~\ref{meagerbutnotImeager}, we show that if $\II\supsetneq\bd_\kappa$, then there is always a meager set that is not $\II$-meager. In Theorem \ref{theorem:indfromreaping}, we show that if $\kappa$ is simple, $\reap=2^\kappa$, and $\II$ is tall, then there is an $\II$-meager set that is not meager.

\medskip

\textbf{Question} (see Question \ref{question:invariants}). Is there always an $\II$-meager set that is not meager, at least if $\kappa$ is simple?

\medskip

Finally, in Section \ref{section:baire}, we investigate the connections between the Baire property in the bounded and in the nonstationary topology. In Proposition \ref{proposition:baire}, we show that there is a subset of $2^\kappa$ that has the Baire property (in the bounded topology),
yet
it does not have the Baire property in the nonstationary topology. In Theorem~\ref{Delta11Baire}, we extend another result about inaccessible cardinals from \cite{fkk}, showing that if $\kappa$ is simple and every $\boldsymbol{\Delta}^1_1$-subset of $2^\kappa$ has the Baire property, then every $\boldsymbol{\Delta}^1_1$-subset of $2^\kappa$ has the Baire property in the nonstationary topology.

\section{On the Borel hierarchy in ideal topologies}
\label{section Borel hierarchy}

\subsection{A normal form for closed sets}\label{section:normalform}

In this short section, we provide a normal form for closed sets in ideal topologies, that generalizes the usual normal form with respect to the bounded topology, and which will be very useful later on.

\medskip

For $x\in 2^\kappa$ and $J\subseteq\mathcal P(\kappa)$, let \[x {\Restr} J:=\{x\restr A\mid A\in J\},\] and, for $P\subseteq\heart_{\mathcal P(\kappa)}$, let \[[P]_J:=\{x\in 2^\kappa\mid x\Restr J\subseteq P\}.\]

Note that if $J=\kappa\subseteq\mathcal P(\kappa)$ and
$T\subseteq\heart_J = 2^{<\kappa}$ is a tree,
then $[T]_J=[T]$ is exactly the body of $T$, i.e., the set of branches (of length $\kappa$) through $T$. Our next result shows that the usual normal form for closed sets as sets of branches through trees generalizes to the context of arbitrary ideal topologies.

\begin{proposition}\label{normal form}
  If $P\subseteq\heart_{\mathcal P(\kappa)}$, then \[[P]_\II=\{x\in 2^\kappa\mid x\Restr\II\subseteq P\}\] is an $\II$-closed subset of $2^\kappa$.

Conversely, if $X\subseteq 2^\kappa$ is $\II$-closed, then there is $P\subseteq\heart_\II$ such that $X=[P]_\II$. Moreover, we may assume that $P$ is closed under restrictions, and that $P$ is \emph{pruned} -- that is, for every $p\in P$, there is $x\in 2^\kappa$ with $p\in x\Restr\II\subseteq P$.
\end{proposition}
\begin{proof}
  Let $P\subseteq\heart_{\mathcal P(\kappa)}$, and let $X=[P]_\II$. If $x\in 2^\kappa$ is not an element of $X$, then there is $A\in\II$ with $x\restr A\not\in P$. But then
  $[x\restr A]$
    is disjoint from $X$, hence the complement of $X$ is $\II$-open, and thus $X$ is $\II$-closed.

  Conversely, assume now that $X\subseteq 2^\kappa$ is $\II$-closed, and let $P=\{x\restr A\mid x\in X\,\land\,A\in\II\}$. Now if $x\in X$, then clearly $x\Restr\II\subseteq P$ by definition of $P$. If $x\not\in X$, then since $X$ is $\II$-closed, there is $A\in\II$ with
  $X\cap[x\restr A]=\emptyset$.
    But then, $x\restr A\not\in P$, hence also $x\Restr\II\not\subseteq P$. Moreover, observe that $P$ is clearly closed under restrictions and pruned.
\end{proof}

\subsection{Tallness, and related properties of ideals}

Tallness is a well-known notion in case $\kappa=\omega$. Generalizations of this concept will turn out to be very useful and natural in the context of ideal topologies.

\begin{definition}\label{definition:tall}
  Let $\II$ and $\JJ$ be ideals on $\kappa$. We say that $\II$ is \emph{$\JJ$-tall} if for every $A\in\JJ^+$, there is $B\subseteq A$ in $\JJ^+\cap\II$. We say that $\II$ is \emph{tall} if it is $\bd_\kappa$-tall. We say that $\II$ is \emph{stationarily tall} if it is $\NS_\kappa$-tall.
\end{definition}

The following is very easy to prove (compare with Proposition~\ref{proposition:NS_sequentially_tall}):
\begin{observation}
$\NS_\kappa$ is tall, and hence every $\II\supseteq\NS_\kappa$ is tall.
\end{observation}

On the other hand, $\II$ being tall does not necessarily
imply that $\II \supseteq \NS_\kappa$.\footnote{As an example, let
$\kappa$ be strongly compact (i.e.,
every ${<}\kappa$-complete
filter on $\kappa$ can be extended to a ${<}\kappa$-complete ultrafilter).
Extend the filter generated by a fixed nonstationary set and the co-bounded sets to a ${<}\kappa$-complete ultrafilter
to get a maximal ideal which does not contain all nonstationary sets; by
Observation~\ref{observation:stationary_tall_etc}(2), each maximal ideal is tall, which
finishes the argument.}
Note that every normal ideal\footnote{Remember that an ideal is \emph{normal} if it is closed under the taking of diagonal unions.} contains the nonstationary ideal, and is thus tall by the above.
Clearly, any $\JJ$-tall ideal contains a set in $\JJ^+$, but this is not sufficient for being $\JJ$-tall. However, there is an easy property which implies $\JJ$-tallness:

\begin{proposition}\label{proposition:stronger_than_J_tall}
If $\II$ contains a set in $\JJ^*$ then $\II$ is $\JJ$-tall.
\end{proposition}
\begin{proof}
If $C\in \II\cap \JJ^*$ and $A\in \JJ^+$, then $C\cap A\subseteq A$ is in $\JJ^+\cap \II$.
\end{proof}

For $x\subseteq\kappa$, let $\chi_x\colon\kappa\to 2$ denote the characteristic function of $x$.
We will always identify $x$ and $\chi_x$.
We may thus view
a collection~$\mathcal K$ of subsets of $\kappa$
as a subset of $2^\kappa$ via the identification
\[
\mathcal K = \{\chi_x\mid x \in \mathcal K \}.
\]

\begin{observation}\label{observation:J_J*_same} Let $H_\one\colon 2^\kappa\rightarrow 2^\kappa$ be the homeomorphism defined as in the proof of
Proposition~\ref{prop:flip}.
Since for any ideal $\JJ$ on $\kappa$, $H_\one[\JJ]=\JJ^*$, it follows that $\JJ$ and $\JJ^*$ have the same topological properties in the $\II$-topology, in particular $\JJ$ is $\II$-open iff $\JJ^*$ is $\II$-open, and $\JJ$ is $\II$-closed iff $\JJ^*$ is $\II$-closed.
\end{observation}

Our next proposition provides a characterization of $\JJ$-tallness of an ideal $\II$ in terms of whether certain sets lie low down in the $\II$-Borel hierarchy.

\begin{proposition}\label{proposition:talliffpositiveopen}
  The following are equivalent:
  \begin{enumerate}
    \item $\II$ is $\JJ$-tall.
    \item $\JJ^+$ is $\II$-open.
    \item $\JJ$ is $\II$-closed.
    \item $\JJ^*$ is $\II$-closed.
  \end{enumerate}
\end{proposition}
\begin{proof}
  \emph{(1) implies (2):}
  Assume that $\II$ is a $\JJ$-tall ideal. Then, \[\JJ^+=\bigcup\{[\one \restr A]\mid A\in\II\cap\JJ^+\}\] is clearly $\II$-open.

\medskip

  \emph{(2) implies (1):} Assume that $\JJ^+$ is $\II$-open, and let $A\in\JJ^+$. Then, $\JJ^+$ contains an $\II$-cone $[f]$ with $A\in[f]$. Since $[f]$ must only contain elements of $\JJ^+$, $f$ has to take value $1$ on some $B\in\JJ^+$ with $B\subseteq\dom(f)\in\II$. This shows that every element of $\JJ^+$ contains an element of $\JJ^+\cap\II$, which means that $\II$ is $\JJ$-tall, as desired.

\medskip

\emph{(2) and (3) are equivalent}, because $\JJ^+$ is the complement of $\JJ$.
The \emph{equivalence of (3) and (4)} follows from Observation~\ref{observation:J_J*_same}.
\end{proof}

The next proposition provides a similar characterization of a property stronger than $\JJ$-tallness (see  Proposition~\ref{proposition:stronger_than_J_tall}).
\begin{proposition}\label{proposition:containedimpliesopen}
  The following are equivalent:
  \begin{enumerate}
    \item $\II$ contains a set in $\JJ^*$.
    \item $\JJ^+$ is $\II$-closed.
    \item $\JJ$ is $\II$-open.
    \item $\JJ^*$ is $\II$-open.
  \end{enumerate}
\end{proposition}
\begin{proof}
  \emph{(1) implies (4):}  Assume that there exists $C$ in $\JJ^*\cap\II$.
  Then, $\JJ^*=\bigcup_{x\in\JJ^*}[\one \restr (x\cap C)]$.

  \medskip

 \emph{(4) implies (1):}   Now assume that $\JJ^*$ is $\II$-open, and hence contains an $\II$-cone $[f]$. If the domain of $f$ were not in $\JJ^*$, then, since $f^{-1}[\{1\}]\in[f]$, we obtain a contradiction. But this implies that $\dom(f)\in\JJ^*\cap\II\ne\emptyset$, as desired.

 \medskip

Again, \emph{(2) and (3) are equivalent}, because $\JJ^+$ is the complement of $\JJ$, and
the \emph{equivalence of (3) and (4)} follows from Observation~\ref{observation:J_J*_same}.
\end{proof}

Note that the two propositions above show that for an ideal $\JJ$ the following implication holds:\footnote{By Observation \ref{observation:J_J*_same}, an analogous remark applies to filters rather than ideals.}
 $$ \JJ \text{ is } \II \text{-open} \Rightarrow \JJ \text{ is } \II \text{-closed}.$$
Of course this can also be shown directly
(the proof makes essential use of the property of an ideal to be closed under
unions).

Let us shed light on a few relationships between
(stationary) tallness and other properties of ideals that we are making use of in this paper:

\begin{observation}\
\label{observation:stationary_tall_etc}
  \begin{enumerate}
    \item
    If $\II$ contains a club subset of $\kappa$, then $\II$ is stationarily~tall.
    \item If $\II$ is a maximal ideal, then $\II$ is both tall and stationarily~tall.
    \item Ideals which contain a stationary subset of $\kappa$ are not necessarily tall or stationarily tall.
    \item Stationarily tall ideals are not necessarily tall.
  \end{enumerate}
\end{observation}
\begin{proof}
  \begin{enumerate}
    \item Immediately follows from Proposition~\ref{proposition:stronger_than_J_tall},
letting $\JJ=\NS_\kappa$.
    \item Let $\II$ be a maximal ideal, and let $A$ be an unbounded subset of $\kappa$ that is not in $\II$. Partition $A$ into two disjoint unbounded subsets $A_0$ and $A_1$ of $\kappa$. By the maximality of $\II$, either $A_0$ or $A_1$ is an element of $\II$, yielding $\II$ to be tall.

More generally, the same proof yields that $\II$ is $\JJ$-tall, provided that every set in~$\JJ^+$ can be
partitioned into two disjoint sets in~$\JJ^+$. Since every stationary set can be partitioned into two disjoint stationary sets, $\II$ is also stationarily tall.

    \item Let $\II$ be the ideal generated by the bounded ideal and a single stationary and co-stationary subset $S$ of $\kappa$. Then, $\II$ is neither tall nor stationarily tall, for the complement of $S$ contains no unbounded subset of $\kappa$ in $\II$.
    \item Let $\II$ be the ideal generated by the bounded ideal and a single club subset $C$ of $\kappa$, the complement of which is unbounded in $\kappa$. Then, $\II$ is stationarily tall by (1), yet
    the complement of $C$ has no unbounded subset in $\II$, showing that $\II$ is not tall.\qedhere
  \end{enumerate}
\end{proof}

\subsection{On the collection of unbounded sets}\label{subsec:ub_kappa}

Using that $\II$ is ${<}\kappa$-complete, the $\II$-open sets are closed under ${<}\kappa$-intersections, and the $\II$-closed sets are closed under ${<}\kappa$-unions. We may define an $\II$-Borel hierarchy as usual in higher descriptive set theory, through using $\kappa$-intersections, $\kappa$-unions and complements. For example, on the second level of this hierarchy, we have the $\II$-$F_\kappa$-sets, which are the $\kappa$-unions
$\bigcup_{\alpha < \kappa} X_\alpha$
of $\II$-closed sets, and the $\II$-$G_\kappa$-sets, which are the $\kappa$-intersections
$\bigcap_{\alpha < \kappa} X_\alpha$
of $\II$-open sets, etc.\footnote{It is easy to see that if $\kappa^\lambda=\kappa$, then the $\II$-$G_\kappa$-sets are closed under $\lambda$-unions, and correspondingly, the $\II$-$F_\kappa$-sets are closed under $\lambda$-intersections. In particular, if $\kappa^{<\kappa}=\kappa$, then these classes are closed under ${<}\kappa$-unions and ${<}\kappa$-intersections, respectively.}

\medskip

Let  \[\ub_\kappa=\{\chi_x\mid x\textrm{ is an unbounded subset of }\kappa\}\]
be the collection of unbounded subsets of $\kappa$.
Note that
$\ub_\kappa = \bigcap_{\alpha<\kappa} X_\alpha$,
where
$X_\alpha = \{ x\in 2^\kappa \mid  x(\beta)=1 \text{ for some } \beta\geq\alpha\}$ is open
(in the bounded topology), hence $\ub_\kappa$ is $G_\kappa$ (again in the bounded topology); therefore, $\ub_\kappa$ is $\II$-$G_\kappa$ for any ideal~$\II$.

The following is an immediate consequence of Proposition \ref{proposition:talliffpositiveopen},
letting $\JJ=\bd_\kappa$:

\begin{corollary}\label{corollary:tall}
$\II$ is tall
if and only if
$\ub_\kappa$ is $\II$-open.
\end{corollary}

Note that $\ub_\kappa$ is not $\II$-closed,
 for any choice of ideal $\II$. We will now show that $\ub_\kappa$ can never be an $\II$-$F_\kappa$ set.
This in particular yields an example that whenever $\II$ is tall, there is an $\II$-open set that is not $\II$-$F_\kappa$
(so there is always an Edinburgh open set that is not Edinburgh $F_\kappa$).
For $\X \subseteq \kappa$, let $\ub(\X)$ be the collection of subsets of $\kappa$ which have unbounded intersection with $\X$.

\begin{theorem}\label{theorem:ubnotfsigma}
  $\ub_\kappa$ is not $\II$-$F_\kappa$. In fact, for
every
$\X\notin \II$,
the collection
$\ub(\X)$
is not $\II$-$F_\kappa$.
\end{theorem}

\begin{proof}
  Let $\X\notin \II$, and assume for a contradiction that $\ub(\X)$ is $\II$-$F_\kappa$, i.e., that $\ub(\X)=\bigcup_{\alpha<\kappa}[P_\alpha]_\II$, with each $P_\alpha\subseteq\heart_\II$ closed under restrictions. We want to inductively construct a set in $\ub(\X)$ which is not in the above union, and thus reach a contradiction. The key ingredient will be the following claim. If $\B\subseteq\kappa$, we say that $f\colon \B\to 2$ is \emph{bounded} (in $\kappa$) if
  $\{\alpha\in \B\mid f(\alpha)=1\}$  is bounded in $\kappa$.\footnote{Clearly, if $\B=\kappa$, then $f\colon \B\to 2$ is bounded in $\kappa$ if and only if $f$ is bounded in $\kappa$ in the usual sense when identified with a subset of $\kappa$.}

\begin{claim*}
  Suppose $f\in\heart_\II$ is bounded, $\alpha<\kappa$, and $[P]_\II$ is $\II$-closed with $P\subseteq\heart_\II$ closed under restrictions, and such that $[P]_\II$ contains only unbounded subsets of $\kappa$. Then there is an extension $g\supseteq f$ of $f$ in $\heart_\II$ which is bounded, such that $g\not\in P$, and such that $g(\gamma)=1$ for some $\gamma\ge\alpha$ in $\X$.
\end{claim*}
\begin{proof}
Let $f^*\in\heart_\II$ be bounded and extending $f$ such that $f^*(\gamma)=1$ for some $\gamma\ge\alpha$ in $\X$ (using that $\X\notin \II$). For $\B\subseteq \kappa\setminus\dom(f^*)$ in $\II$, let $f^*_\B\in\heart_\II$ denote the extension of $f^*$ with $\dom(f^*_\B)=\dom(f^*)\cup \B\in\II$ and with $f^*_\B(\alpha)=0$ for every $\alpha\in \B$. Now, assume for a contradiction that every such $f^*_\B$ were an element of $P$. But then, letting $x\in 2^\kappa$ be the extension of $f^*$ with $x(\alpha)=0$ for every $\alpha\in\kappa\setminus\dom(f^*)$, it follows, using that $P$ is closed under restrictions, that $x\Restr\II\subseteq P$, and hence that $x\in[P]_\II$. But $x$ is a bounded subset of $\kappa$, contradicting our assumption on $P$. Hence, we may pick $g=f^*_\B$, for some $\B$ as above, for which $f^*_\B\not\in P$.
\end{proof}
Let $f_0=\emptyset$. Given $f_\alpha$, let $f_{\alpha+1}$ be a bounded extension of $f_\alpha$ in $\heart_\II$ with $f_{\alpha+1}(\gamma)=1$ for some $\gamma\ge\alpha$ in $\X$ and with $f_{\alpha+1}\not\in P_\alpha$, by an application of the claim. At limit stages $\alpha<\kappa$, let $f_\alpha=\bigcup_{\beta<\alpha}f_\beta\in\heart_\II$.
Let $f\in 2^\kappa$ be any extension of $\bigcup_{\alpha<\kappa}f_\alpha$. Then $f\in \ub(\X)$,
yet $f\notin \bigcup_{\alpha<\kappa}[P_\alpha]_\II$, which yields our desired contradiction.
\end{proof}

Let us say that $\II$ is
\emph{tall on~$\X$} if
for every unbounded set~$A' \subseteq \X$, there is an unbounded set $\B \subseteq A'$ with $\B \in \II$.

\begin{corollary}\label{corollary:opennotfsigma}
If $\X\notin \II$ and $\II$ is tall on $\X$, then $\ub(\X)$ is $\II$-open,
yet not $\II$-$F_\kappa$.
\end{corollary}

\begin{proof}
Note that $\ub(\X)$ is $\II$-open, because
  \[\ub(\X)=\bigcup\{[\one \restr \B]\mid \B\in \II \textrm{ is an unbounded subset of }\X\}.\]
  By Theorem~\ref{theorem:ubnotfsigma}, $\ub(\X)$ is not $\II$-$F_\kappa$, as desired.
\end{proof}

So  whenever $\II$ is tall on a set which is not in $\II$,
there exists an $\II$-open set which is not $\II$-$F_\kappa$.

\subsection{The lowest levels of the ideal Borel hierarchies}\label{section:lowlevels}

Corollary \ref{corollary:opennotfsigma} clearly motivates the question on how the lowest levels of the $\II$-Borel hierarchy are related
 when $\II$ is tall on a set $\X\notin \II$, given the unusual non-implication from being $\II$-open to being $\II$-$F_\kappa$.
  We will consider some of the lowest level natural classes of $\II$-Borel sets: $\II$-open sets, $\II$-closed sets, $\II$-$G_\kappa$-sets, and $\II$-$F_\kappa$-sets,
as well as the
class
\[
\II\textrm{-}\cap = \{ X \cap Y \mid X \textrm{ is }\II\textrm{-open and } Y \textrm{ is }\II\textrm{-closed}\}
\]
of intersections of $\II$-open and $\II$-closed sets,
and the class
\[
\II\textrm{-}\cup = \{ X \cup Y \mid X \textrm{ is }\II\textrm{-open and } Y \textrm{ is }\II\textrm{-closed}\}
\]
of unions of $\II$-open and $\II$-closed sets. Note that the class $\II$-$\cap$ is the second level of the difference hierarchy
over the $\II$-open sets (or the $\II$-closed sets),
and the class $\II$-$\cup$ is its dual.

We will show that, assuming that $\II$ is tall on a set $\X\notin \II$,
\emph{only} the following six trivial implications between the above hold (which hold by the very definition of the classes involved):
\begin{itemize}
  \item $\II$-open sets are $\II$-$G_\kappa$, $\II$-closed sets are $\II$-$F_\kappa$.
  \item $\II$-open sets are both $\II$-$\cap$ and $\II$-$\cup$.
  \item $\II$-closed sets are both $\II$-$\cap$ and $\II$-$\cup$.
\end{itemize}

Let $\I_\kappa=\{\chi_{\{\alpha\}}\mid\alpha<\kappa\}$.
Moreover, fix a set $\X\notin \II$ such that $\II$ is tall on $\X$.
We start by determining exactly to which of the above
classes some of our major examples belong to -- these are marked with a + in the following table. The entries of this table are immediate for $2^\kappa$, and they follow for $\ub(\X)$ by Corollary \ref{corollary:opennotfsigma}.
We will provide the easy verifications of the other entries in Lemma \ref{lemma:basicexamples} below.

\bigskip

\begin{table}[h!]
\centering
\begin{tabular}
{ c | c | c | c | c | c | c }
  & $\II$-open & $\II$-closed & $\II$-$\cap$ & $\II$-$\cup$ & $\II$-$G_\kappa$ & $\II$-$F_\kappa$ \\
\hline
  $2^\kappa$ & + & + & + & + & + & +\\
  $\{\zero\}$ &  & + & + & + & + & +\\
  $\I_\kappa$ &  &  & + &  & + & +\\
  $\ub(\X)$ & + &  & + & + & + &
\end{tabular}
\bigskip
\caption{Borel properties of some basic sets}
\label{table:1}
\end{table}

\begin{lemma}\label{lemma:basicexamples}\
  \begin{enumerate}
    \item $\{\zero\}$ is not $\II$-open,
    yet $\II$-closed and $\II$-$G_\kappa$.
    \item $\I_\kappa$ is neither $\II$-open nor $\II$-closed,
    yet $\II$-$\cap$.
      Moreover, it is not $\II$-$\cup$,
      yet it is both $\II$-$G_\kappa$ and $\II$-$F_\kappa$.
  \end{enumerate}
\end{lemma}
\begin{proof}
The properties of $\{\zero\}$ and of $\I_\kappa$ marked with a $+$ in the above table correspond exactly to the respective properties that those sets have in the bounded topology. Whenever a set has such a property in the bounded topology, it inherits to the $\II$-topology: for example, $\{\zero\}$ is closed in the bounded topology, and hence it is $\II$-closed.

For~(1), it remains to show that $\{\zero\}$ is not $\II$-open, which follows as every non-empty $\II$-open set has size~$2^\kappa$. The same argument shows that $\I_\kappa$ is not $\II$-open.
To finish the proof of~(2), it remains to show that $\I_\kappa$ is neither $\II$-closed nor $\II$-$\cup$.
If $\I_\kappa$ were $\II$-closed, its complement would be $\II$-open, hence would contain an $\II$-cone around~$\zero$; but each such $\II$-cone contains an element $\chi_{\{\alpha\}}$ of $\I_\kappa$, a contradiction.
Moreover, if $\I_\kappa$
were
in $\II$-$\cup$, i.e., the union of an $\II$-open and an $\II$-closed set, the above cardinality argument yields that the $\II$-open part is in fact empty, which is impossible by the above.
\end{proof}

Note that
sets with further patterns with respect to the properties in Table~\ref{table:1} can simply be generated by the taking of complements: if $X \subseteq 2^\kappa$ and $Y = 2^\kappa \setminus X$ is the complement of $X$, then $Y$ is $\II$-closed if and only if $X$ is $\II$-open, and vice versa, and correspondingly for the properties $\II$-$\cap$ and $\II$-$\cup$, as well as the properties $\II$-$G_\kappa$ and $\II$-$F_\kappa$. In particular, this allows us to obtain all the five possible patterns for sets which are either $\II$-open or $\II$-closed, by also considering the complements of $\{\zero\}$ and of
$\ub(\X)$.

\medskip

We will show that for sets which are neither $\II$-open nor $\II$-closed, all 16 combinations of the remaining four properties occur (see Table~\ref{table:16}), very much unlike
the case of the bounded topology. The combination of the above then shows that there are no implications between
(any combinations of)
the six properties in Table~\ref{table:1} other than the trivial ones listed above. All of our examples below will be based on the basic sets $2^\kappa$, $\{\zero\}$, $\I_\kappa$ and $\ub(\X)$. We will use the method of taking \emph{unions on disjoint cones}:

\begin{definition}
  Given $X,Y\subseteq 2^\kappa$ we say that $W=X\dot\cup Y$ is \emph{a union of $X$ and $Y$ on disjoint $\II$-cones} in case $[f]$ and $[g]$ are two disjoint $\II$-cones, which are both, using the respective induced topologies, homeomorphic to $2^\kappa$ with the $\II$-topology, via bijections $\pi\colon 2^\kappa\to[f]$ and $\rho\colon2^\kappa\to[g]$, and such that $W=\pi[X]\cup\rho[Y]$. Unions $X\dot\cup Y\dot\cup Z$ of three (or more) sets are defined analogously.\footnote{Note that by Observation \ref{observation:homeomorphiccones}, we always have at least $\kappa$-many disjoint $\II$-cones available that are each homeomorphic to $2^\kappa$ with the $\II$-topology.}
\end{definition}

\begin{lemma}
  Assume that $W=X\dot\cup Y$ is a union of $X$ and $Y$ on disjoint $\II$-cones. Then, for each in the following list of
  classes, $W$ is a member if and only if both $X$ and $Y$ are members -- these classes are: $\II$-open, $\II$-closed, $\II$-$\cap$, $\II$-$\cup$, $\II$-$G_\kappa$, $\II$-$F_\kappa$. An analogous result holds for unions on disjoint $\II$-cones of a larger (finite will be sufficient for our purposes) number of sets.
\end{lemma}
\begin{proof}
  An easy check that we would like to leave to our readers.
\end{proof}

Armed with the above lemma, we may now easily construct sets which are neither $\II$-open nor $\II$-closed, and satisfy arbitrary combinations of the remaining four properties of being $\II$-$\cap$, $\II$-$\cup$, $\II$-$G_\kappa$ and $\II$-$F_\kappa$, by simply combining the four basic sets $2^\kappa$, $\{\zero\}$, $\I_\kappa$ and $\ub(\X)$ from above via taking unions on disjoint $\II$-cones, and via taking complements of such sets. We illustrate those results in Table \ref{table:16} below. For a set $X$, we let $\overline{X}$ denote the complement of $X$: for example, $\overline{\{\zero\}}$ denotes $2^\kappa\setminus\{\zero\}$ in the table below. When we put the symbol $\sim$ in our table, this means that a set with exactly the properties indicated by the combination of
$+$'s
in its row can simply be obtained by considering the complement of one of the other sets used in the table (which might only appear further down in the table) -- note that, as we mentioned above, a set is in $\II$-$\cap$ if and only if its complement is in $\II$-$\cup$, and it is in
$\II$-$G_\kappa$ if and only if its complement is in $\II$-$F_\kappa$, i.e., taking complements corresponds to switching the entries of the corresponding columns in the table.
We will leave the completely straightforward task of verifying any of the entries in Table \ref{table:16} to our interested readers.

\renewcommand{\arraystretch}{1.2}
\begin{table}[hbt!]
\centering
\begin{tabular}
{ c | c | c | c | c }
  & $\II$-$\cap$ & $\II$-$\cup$ & $\II$-$G_\kappa$ & $\II$-$F_\kappa$ \\
\hline
\hline
 $\{\zero\}\dot\cup\overline{\{\zero\}}$ & + & + & + & + \\
 $\ub(\X)\dot\cup\{\zero\}$ & + & + & + &  \\
 $\sim$ & + & + &  & + \\
 $\ub(\X)\,\dot\cup\,\overline{\ub(\X)}$ & + & + &  &  \\
\hline
 $\I_\kappa$ & + &  & + & + \\
 $\I_\kappa\,\dot\cup\,\ub(\X)$ & + &  & + &  \\
 $\sim$ & + &  &  & + \\
 $\I_\kappa\,\dot\cup\,\ub(\X)\,\dot\cup\,\overline{\ub(\X)}$ & + &  &  &  \\
\hline
 $\sim$ &  & + & + & + \\
 $\ub(\X)\,\dot\cup\,\overline{\I_\kappa}$ &  & + & + &  \\
 $\sim$ &  & + &  & + \\
 $\sim$ &  & + &  &  \\
\hline
 $\I_\kappa\,\dot\cup\,\overline{\I_\kappa}$ &  &  & + & + \\
 $\ub(\X)\,\dot\cup\,\I_\kappa\,\dot\cup\,\overline{\I_\kappa}$ &  &  & + &  \\
 $\sim$ &  &  &  & + \\
 $\ub(\X)\,\dot\cup\,\overline{\ub(\X)}\,\dot\cup\,\I_\kappa\,\dot\cup\,\overline{\I_\kappa}$ &  &  &  &  \\
\hline
\end{tabular}
\bigskip
\caption{All 16 properties}
\label{table:16}
\end{table}

It remains open whether there is an $\II$-Borel hierarchy which is somewhat similar to the classical Borel hierarchy.

\begin{question}\label{question Borel hierarchy}
 Given that $\II$ contains an unbounded subset of $\kappa$,
 are there $\II$-Borel sets which are substantially more complicated than the examples from Table~\ref{table:16}?
 Do the $\II$-Borel sets even form a strict hierarchy of length $\kappa^+$ (at least if $\II=\NS_\kappa$)?
\end{question}

\subsection{On the collection of closed and unbounded sets}\label{section:clubs}

Let \[\Closed_\kappa=\{\chi_x\mid x\textrm{ is a closed subset of }\kappa\}\]
be the collection of closed (possibly bounded) subsets of $\kappa$, which is a closed set in the bounded topology
(and therefore also in any ideal topology). Let
 \[\Club_\kappa=\{\chi_x\mid x\textrm{ is a club subset of }\kappa\}=\Closed_\kappa\,\cap\,\ub_\kappa\]
be the collection of club subsets of $\kappa$.
Recall from Section~\ref{subsec:ub_kappa} that
$\ub_\kappa$ is $G_\kappa$ (in the bounded topology);
moreover, since
every closed set (in the bounded topology) is $G_\kappa$ (again in the bounded topology),
also
$\Closed_\kappa$ is
$G_\kappa$ (in the bounded topology);
therefore,
$\Club_\kappa$ is $G_\kappa$ (in the bounded topology),
and hence $\Club_\kappa$ is $\II$-$G_\kappa$ for any ideal~$\II$.

If $\II$ is tall (hence in particular if $\II = \NS_\kappa$),
Corollary~\ref{corollary:tall}
implies that
$\Club_\kappa$ is, unlike for the bounded topology, an intersection of an $\II$-open and a closed (and hence also $\II$-closed) set (so in this case,
$\Club_\kappa$ is~$\II$-$\cap$).

We now want to deal with the question when $\Club_\kappa$ can be on any of the other low levels of $\II$-Borel hierarchies. We first characterize exactly when $\Club_\kappa$ is $\II$-closed:

\begin{observation}\label{observation:stationaryclubclosed}
  $\Club_\kappa$ is $\II$-closed if and only if $\II$ contains a stationary subset of $\kappa$.
\end{observation}
\begin{proof}
Fix $S\in\II$ stationary. Let $x\subseteq\kappa$ not be in $\Club_\kappa$, i.e., $x$ not closed unbounded.
In case~$x$ is not closed, let $\alpha < \kappa$ be such that $x \restr \alpha$ is not closed; then $[x \restr \alpha] \cap \Club_\kappa = \emptyset$.

If $x$ is bounded, fix $\alpha < \kappa$ such that $x \subseteq \alpha$, and let $S' := S\setminus\alpha\in\II$. Since $S'$ is stationary, it intersects each closed unbounded subset of $\kappa$, and hence $[x\restr S']=[\zero\restr S']$ has empty intersection with $\Club_\kappa$.

The above shows that in each case, $x$ is in the $\II$-interior of the complement of $\Club_\kappa$, and hence that $\Club_\kappa$ is $\II$-closed, as desired.

For the reverse direction, assume that $\Club_\kappa$ is $\II$-closed. Then, its complement contains an $\II$-cone $[f]$ with $\emptyset\in[f]$. Hence, $f$ has to have constant value $0$. Since $[f]$ must not contain a club subset of $\kappa$,
$\dom(f)$ has to be stationary, and $\dom(f) \in \II$,
which finishes the proof.
\end{proof}

Note that if $\Club_\kappa$ has non-empty $\II$-interior, then $\II$ contains a club, which can be seen as follows.
Assume $C\in \Club_\kappa$ is in the $\II$-interior, i.e., there exists $f\in \heart_\II$ with $C\in [f]\subseteq \Club_\kappa$. If $\dom(f)$ does not contain a club, then there exists $x\in [f]$ which is not a club.
By the above Observation~\ref{observation:stationaryclubclosed}, it follows that $\Club_\kappa$ is $\II$-closed whenever $\Club_\kappa$ has non-empty $\II$-interior.
From this, we
get that $\Club_\kappa$ being $\II$-$\cup$ is equivalent to $\Club_\kappa$ being $\II$-closed.

In particular, $\Club_\kappa$ being $\II$-open implies that $\II$ contains a club (and hence, by Observation~\ref{observation:stationaryclubclosed}, that $\Club_\kappa$ is $\II$-closed). We now
investigate the possibility of $\Club_\kappa$ being $\II$-open and
give an exact characterization.
Let $\Lim$ denote the club set of all limit ordinals in~$\kappa$.

\begin{proposition}\label{proposition:clubopen}
  $\Club_\kappa$ is $\II$-open if and only if the following Property (*) holds: $\Lim\in\II$, and for every nonstationary subset $N$ of $\Lim$, there is a regressive function $r\colon N\to\kappa$ such that $\bigcup_{\alpha\in N}[r(\alpha),\alpha)\in\II$.
\end{proposition}

\begin{proof}
Assume that $\Club_\kappa$ is $\II$-open. Let $[f]$ be an $\II$-cone such that $\kappa\in [f]\subseteq\Club_\kappa$. Then, $f=\one\restr A$ for some $A\in\II$. If $\alpha\in\Lim\setminus A$, then $[f]$ contains a subset of $\kappa$ that does not contain $\alpha$ as an element,
yet is unbounded below $\alpha$, contradicting that $[f]\subseteq\Club_\kappa$, and thus showing that $A\supseteq\Lim\in\II$.

Let $N$ be a nonstationary subset of $\Lim$. Let $C\subseteq\Lim$ be a club that is disjoint from $N$. There is some $[f]\subseteq\Club_\kappa$ with $C\in[f]$, hence $f=C\restr A$ for some $A\in\II$.
If for some $\alpha\in N$, the complement of $A$ were unbounded
in~$\alpha$,
then $[f]$ again contains a subset of $\kappa$ which does not contain $\alpha$ as an element (due to $C \cap N = \emptyset$),
yet is unbounded in $\alpha$, which is again a contradiction.
This now allows us to construct a regressive function $r$
on $N$ that is as desired.

\medskip

Assume now that $\II$ satisfies Property (*). We want to show that $\Club_\kappa$ is $\II$-open. Let $C\subseteq\kappa$ be any club subset of $\kappa$. It suffices to find a function $f\colon A\to 2$ in $\heart_\II$ such that $C\in[f]\subseteq\Club_\kappa$.

  Let $N$
be the nonstationary set $\Lim\setminus C$, let $r\colon N\to\kappa$ be regressive such that $A':=\bigcup_{\alpha\in N}[r(\alpha),\alpha)\in\II$, and let $A=\Lim\cup A'$. Let $f=C\restr A$. Then, clearly, $C\in[f]$. We have to show that $[f]\subseteq\Club_\kappa$. Let $x\in[f]$.

Since $C \cap A$ is unbounded, $x$ is clearly unbounded. It remains to show that $x$ is closed.
  Take any strictly increasing sequence $\langle\alpha_i\mid i<\cof(\alpha)\rangle$ with limit $\alpha$. The only problematic case that we have to consider is if $\alpha\not\in x$, 
  yet for all $i<\cof(\alpha)$, $\alpha_i\in x$.
Since $x$ and $C$ agree on $A \supseteq \Lim$ and $\alpha \in \Lim$, we have $\alpha \notin C$, hence $\alpha \in N$. So by the definition of $A$, $[r(\alpha), \alpha) \subseteq A$, hence all but boundedly many $\alpha_i$ are in $C$, contradicting $C$ being closed.
\end{proof}

It remains to observe that ideals satisfying Property (*) actually exist:

\begin{observation}
  There is an ideal $\II$ such that $\II$ satisfies Property (*).
\end{observation}
\begin{proof}
  Given $A\subseteq\kappa$, let $A^\oplus=\{\alpha+1\mid\alpha\in A\}$, and let $A^\ominus=\{\alpha-1\mid\alpha\in A$ is a successor ordinal$\}$. Let $\II$ be the (${<}\kappa$-complete) ideal generated by $\Lim$ together with $\{A^\oplus\mid A\subseteq\kappa$ nonstationary$\}$. It is easy to see that by the ${<}\kappa$-completeness of $\NS_\kappa$, $\II$ is a proper ideal on $\kappa$. Let $N$ be a nonstationary subset of $\Lim$, let $C\subseteq\Lim$ be a club that is disjoint from $N$, and let \[A=\bigcup\{[\alpha+2,\beta)\mid\beta\in N\,\land\,\alpha=\max(C\cap\beta)\}.\]
Let $B=A^\ominus$. Then, $B$ is disjoint from $C$, and therefore nonstationary, yielding that $A\in\II$, and showing that $\II$ satisfies Property (*).
\end{proof}

We already observed in the above that $\Club_\kappa$ is
Edinburgh~$G_\kappa$, and also that it is an intersection of an Edinburgh open and an Edinburgh closed set; moreover, we have seen that $\Club_\kappa$ is neither Edinburgh open or Edinburgh closed, nor a union of an Edinburgh open and an Edinburgh closed set.
Let us finally verify that $\Club_\kappa$ is not Edinburgh $F_\kappa$
(and hence $\Club_\kappa$ provides a natural example of an Edinburgh~$G_\kappa$ set which is not Edinburgh~$F_\kappa$),
by an argument that builds on the argument for Theorem~\ref{theorem:ubnotfsigma}.
Recall from its proof that
we say that
$f\colon A\to 2$ is bounded if
$\{\alpha\in A\mid f(\alpha)=1\}$ is bounded in $\kappa$.
We will also need the following.

\begin{definition}
If $A\subseteq\kappa$ and $f\colon A\to 2$, we say that $f$ is \emph{closed} if
$\{\alpha\in A\mid f(\alpha)=1\}$ is
a closed subset of~$\kappa$,\footnote{Clearly, if $A=\kappa$, then $f\colon A\to 2$ is closed if and only if $f$ is closed in the usual sense when identified with a subset of $\kappa$.}
i.e., if the following holds:
whenever $\lambda<\kappa$ and $\langle\alpha_i\mid i<\lambda\rangle\subseteq\dom(f)$
  is an increasing sequence with $f(\alpha_i)=1$ for every $i<\lambda$,
then $\alpha=\bigcup_{i<\lambda}\alpha_i\in\dom(f)$ and $f(\alpha)=1$.
\end{definition}

\begin{theorem}\label{theorem:clubnotfsigma}
  $\Club_\kappa$ is not Edinburgh $F_\kappa$.
\end{theorem}
\begin{proof}
Let $\heart$ abbreviate $\heart_{\NS_\kappa}$.
  Assume for a contradiction that
  $\Club_\kappa$ is Edin\-burgh $F_\kappa$,
  i.e., that $\Club_\kappa=\bigcup_{\alpha<\kappa}[P_\alpha]_{\NS_\kappa}$, with each $P_\alpha\subseteq\heart$ closed under restrictions. We want to inductively construct a club subset of $\kappa$ which is not in the above union, and thus reach a contradiction. The key ingredient will be the following claim:

\begin{claim*}
Suppose $f \in \heart$ is closed and bounded,
$C\subseteq\kappa$ is a club subset of $\kappa$ that is disjoint from $\dom(f)$, and $[P]_{\NS_\kappa}$ is Edinburgh closed with $P\subseteq\heart$ closed under restrictions,
and such that $[P]_{\NS_\kappa}$ contains only
club
subsets of $\kappa$. Then there is an extension $g\supseteq f$ of $ f$ in $\heart$ which is closed and bounded, such that $g\not\in P$, and such that for some $\gamma\in C$, $\dom(g) \supseteq \gamma + 1$, and
\[\forall\delta\in\dom(g)\setminus\dom(f)\quad g(\delta)=1\iff\delta=\gamma.\]
\end{claim*}

\begin{proof}
  Let $C^* := C\setminus\{\min(C)\}$. Let $f^*\in\heart$ be the extension of $f$ with $\dom(f^*)=\kappa\setminus C^*$, with $f^*(\min(C))=1$, and with $f^*(\alpha)=0$ whenever $\alpha\in\kappa\setminus(\dom(f)\cup C)$. For $A\subseteq C^*$ in $\NS_\kappa$, let $f_A\in\heart$ denote the extension of $f^*$ with $\dom(f_A)=\dom(f^*)\cup A$ and with $f_A(\alpha)=0$ for every $\alpha\in A$. Assume for a contradiction that every such $f_A$ were an element of $P$. But then, letting $x\in 2^\kappa$ be the extension of $f^*$ with $x(\alpha)=0$ for every $\alpha\in C^*$, it follows, since $P$ is closed under restrictions,
  that $x\Restr\NS_\kappa\subseteq P$, and hence that $x\in[P]_{\NS_\kappa}$.
  But\footnotemark{} $x\in\NS_\kappa$, contradicting our assumption on $P$.
\end{proof}

\footnotetext{It is easy to see that $x$ is actually bounded, so assuming in the claim that $[P]_{\NS_\kappa}$ contains only unbounded sets would be sufficient. Therefore a slight modification of the proof of the
theorem
yields the following stronger result: every Edinburgh $F_\kappa$ set which contains $\Club_\kappa$ also contains some bounded set.}

Let $f_0 :=\emptyset$, and let $C_0:=\kappa$. Then
$f_0\in\heart$ is closed and bounded.
Given
$f_\alpha\in \heart$ closed and bounded
and $C_\alpha$ which is disjoint from $\dom(f_\alpha)$,
  let $f_{\alpha+1}$ and $\gamma_{\alpha+1}$ be obtained by an application of the claim with respect to $C_\alpha$ and $P_\alpha$, that is,
$\gamma_{\alpha+1} \in C_\alpha$,
and
$f_{\alpha+1}$
is an extension of $f_\alpha$ in $\heart$ which is closed and bounded, with $\dom(f_{\alpha+1}) \supseteq \gamma_{\alpha+1}+1$ and
\[\forall\delta\in\dom(f_{\alpha+1})\setminus \dom(f_\alpha)\quad f_{\alpha+1}(\delta)=1\iff\delta=\gamma_{\alpha+1},\] and such that $f_{\alpha+1}\not\in P_\alpha$. Let $C_{\alpha+1}\subseteq C_\alpha\setminus(\gamma_{\alpha+1}+1)$ be a club subset of $\kappa$ that is disjoint from $\dom(f_{\alpha+1})$. At limit stages $\alpha<\kappa$, let $\gamma_\alpha := \bigcup_{\beta<\alpha}\gamma_\beta$,
and let $f_\alpha :=\bigcup_{\beta<\alpha}f_\beta\cup\{(\gamma_\alpha,1)\}$. Note that $\gamma_\alpha\in C_\alpha:=\bigcap_{\beta<\alpha}C_\beta$, and therefore $\gamma_\alpha\not\in\dom(f_\beta)$ for any $\beta<\alpha$. Hence,
$f_\alpha$ is indeed a function, and thus an element of $\heart$.
It is easy to see
that the $\gamma_\beta$'s are strictly increasing,
and that $f_\alpha$ is closed and bounded.
Then $f:=\bigcup_{\alpha<\kappa}f_\alpha \in 2^\kappa$ is closed and unbounded (in fact, the characteristic function of the club set $\{ \gamma_\alpha \with \alpha < \kappa \}$),
yet $f\not\in\bigcup_{\alpha<\kappa}[P_\alpha]_{\NS_\kappa}$.
\end{proof}

This yields
yet another characterization of when $\II$ contains a stationary subset of $\kappa$ (which also shows, by
Observation~\ref{observation:stationaryclubclosed}, that $\Club_\kappa$ is $\II$-$F_\kappa$ if and only if it is $\II$-closed):

\begin{corollary}
  $\Club_\kappa$ is $\II$-$F_\kappa$ if and only if $\II$ contains a stationary subset of $\kappa$.
\end{corollary}

\begin{proof}
  If $\II$ contains a stationary subset of $\kappa$, then $\Club_\kappa$ is $\II$-closed by Observation \ref{observation:stationaryclubclosed}, and hence it is trivially also $\II$-$F_\kappa$. On the other hand, $\Club_\kappa$ is not Edinburgh~$F_\kappa$ by Theorem \ref{theorem:clubnotfsigma}, and hence if $\II\subseteq\NS_\kappa$, it is not $\II$-$F_\kappa$, for the Edinburgh topology then refines the $\II$-topology.
\end{proof}

\subsection{The club filter}\label{section:clubfilter}

Let $\cluf_\kappa$ denote the club filter on $\kappa$, i.e., the collection of all subsets of $\kappa$ that contain a club subset of $\kappa$. In the bounded topologies on higher cardinals, the club filter is usually the standard example for a non-Borel set. The situation is somewhat different for $\II$-topologies. The following is an immediate consequence of
Proposition \ref{proposition:talliffpositiveopen} and of Proposition \ref{proposition:containedimpliesopen},
letting $\JJ=\NS_\kappa$:

\begin{corollary} \
\begin{itemize}
\item $\II$ is stationarily tall if and only if $\cluf_\kappa$ is $\II$-closed.
  \item $\II$ contains a club subset of $\kappa$ if and only if $\cluf_\kappa$ is $\II$-open.
  \end{itemize}
\end{corollary}

Compare the first item with Observation~\ref{observation:stationaryclubclosed}, which gives a similar characterization of $\Club_\kappa$ being $\II$-closed. In particular, it follows that
$$ \cluf_\kappa \text{ is } \II\text{-closed} \Rightarrow \Club_\kappa \text{ is } \II\text{-closed}.$$

However, if $\II$ is not stationarily tall, then the situation is somewhat less unusual, for then we will show that
certain relativized club filters are not $\II$-Borel.
In case $\II$ does not contain a stationary set (in particular, if $\II = \NS_\kappa$), the club filter itself turns out to be
not $\II$-Borel.
We do not know
the complexity of the club filter in case
$\II$ is not stationarily tall
yet contains a stationary set.
The following argument extends and generalizes \cite[Theorem 4.2]{halkoshelah}, and also builds on the proof of that theorem.\footnote{In retrospect, we realized that also the arguments for the proofs of Theorem~\ref{theorem:ubnotfsigma} and of Theorem~\ref{theorem:clubnotfsigma} are somewhat similar to the arguments in the proof of \cite[Theorem 4.2]{halkoshelah}.}

\medskip

Let us first recall some basic topological concepts for $\II$-topologies.
A set $X$ is \emph{$\II$-nowhere dense} if for any $\II$-cone $[f]$ there is an $\II$-cone $[g]\subseteq [f]$ with $X\cap [g]=\emptyset$.
A set is \emph{$\II$-meager} if it is a $\kappa$-union of $\II$-nowhere dense sets.
 A set is \emph{$\II$-comeager} if its complement is $\II$-meager.
A set $X$ has the \emph{$\II$-Baire property}, if there is an $\II$-open set $U$ such that $X\Delta U$ is $\II$-meager.
Let us observe that, as usual, every set with the $\II$-Baire property is either $\II$-meager, or is $\II$-comeager in an $\II$-cone. Furthermore, by the usual argument, every $\II$-Borel set has the $\II$-Baire property.

\medskip

Given a stationary set~$S\subseteq\kappa$, let
\[\cluf_\kappa^S=\{A \subseteq \kappa \mid\exists C\subseteq\kappa\textrm{ club with }A\supseteq C\cap S\}.\]

\begin{theorem}\label{theorem:cfnotbaire}
Assume that $S$ is a stationary subset of $\kappa$, and that $\II$ contains no stationary subset of $S$. Then, $\cluf_\kappa^S$ does not have the $\II$-Baire property.
\end{theorem}

\begin{proof}
Towards a contradiction, suppose that $\cluf_\kappa^S$ has the $\II$-Baire property.  First assume that $\cluf_\kappa^S$ is $\II$-meager. Let $\vec{U}=\langle U_i\mid i<\kappa\rangle$ be a sequence of $\II$-open dense sets whose intersection $U:=\bigcap_{i<\kappa}U_i$ is disjoint from $\cluf_\kappa^S$.
We construct sequences $\vec{f}=\langle f_j\mid j<\kappa\rangle$ in $\heart_\II$, $\vec{C}=\langle C_j\mid j<\kappa\rangle$ in $\club_\kappa$, and $\vec{\alpha}=\langle \alpha_j\mid j<\kappa\rangle$ in $\kappa$ with the following properties:
\begin{enumerate}
\item
$f_i\subseteq f_j$, $C_i\supseteq C_j$ and $\alpha_i<\alpha_j$ for all $i<j<\kappa$,
\item
 \begin{enumerate}
\item
$\alpha_j=\min(C_j)$ for all $j<\kappa$,
\item
$\alpha_\lambda=\sup_{i<\lambda}\alpha_i$ for limits $\lambda<\kappa$,
 \end{enumerate}
\item
 \begin{enumerate}
\item
$[f_{j+1}]\subseteq U_j$ for all $j<\kappa$,
\item
$\dom(f_j)\cap C_j \cap S =\emptyset$ for all $j<\kappa$,  and
\item
$f_j(\alpha_i)=1$ for all $i<j<\kappa$ with
$\alpha_i\in S$.
 \end{enumerate}
\end{enumerate}
The construction proceeds as follows.

\begin{enumerate}
\item[(i)]
Choose $f_0\in \heart_\II$ arbitrary, let $C_0$ be a club disjoint from $\dom(f_0)\cap S$, and let $\alpha_0:=\min(C_0)$.
\item[(ii)]
For successors $j+1$, assume that $f_j$ has been constructed. If $\alpha_j\in S$, let $f_j':=f_j\cup\{(\alpha_j,1)\}$, and let $f_j':=f_j$ otherwise. Since $U_{j}$ is $\II$-open dense, there is some $f_{j+1}$ extending $f_j'$ with $[f_{j+1}]\subseteq U_{j}$. Find a club $C_{j+1}\subseteq C_j\setminus (\alpha_j+1)$ disjoint from $\dom(f_{j+1})\cap S$ and let $\alpha_{j+1}:=\min(C_{j+1})$.
\item[(iii)]
For limits $\lambda<\kappa$, let $f_\lambda:=\bigcup_{i<\lambda} f_i$, $C_\lambda:=\bigcap_{i<\lambda}C_i$ and $\alpha_\lambda:=\min(C_\lambda)$.
\end{enumerate}
Then, $f:=\bigcup_{i<\kappa} f_i$ is constant with value $1$ on the intersection of the club $C:=\{\alpha_i\mid i<\kappa\}$ with $S$ by (3)(c).
Since $[f_{j+1}]\subseteq U_j$ for all $j<\kappa$ by (3)(a), any total extension $g\supseteq f$ contradicts that $\cluf_\kappa^S$ is disjoint from $U$.

Finally, assume that $\cluf_\kappa^S$ is $\II$-comeager in $[h]$ for some $h\in \heart_\II$.
But then, virtually the same construction, letting $f_0=h$ and setting
$f(\alpha_j)=0$ (instead of $f(\alpha_j)=1$) in case $\alpha_j\in S$, yields a contradiction just as in the previous case.
\end{proof}

\begin{corollary} \
\label{club filter is no Edinburgh Borel}
  \begin{enumerate}
    \item
    $\cluf_\kappa$ does not have the Edinburgh Baire property, hence is not Edinburgh~Borel.
    \item $\II$ is stationarily tall if and only if for every stationary $S\subseteq\kappa$, $\cluf_\kappa^S$ is $\II$-Borel.
 \item If $\II$ is not stationarily tall, then there is a set without the $\II$-Baire property.
    \item If $\II$ is not stationarily tall, then there is a set which is not $\II$-Borel.
  \end{enumerate}
\end{corollary}
\begin{proof}
  The above are immediate from
Theorem~\ref{theorem:cfnotbaire}
  and the comment preceding it that $\II$-Borel sets have the $\II$-Baire property, except that we still have to verify for (2) that if $\II$ is stationarily tall, then for every stationary set~$S\subseteq\kappa$, $\cluf_\kappa^S$ is $\II$-Borel (in fact: $\II$-closed).
  Let $x \notin \cluf_\kappa^S$. Then
  for all clubs $ C\subseteq \kappa$, $(\kappa\setminus x)\cap C\cap S\neq\emptyset$,
  so $(\kappa\setminus x) \cap S$ is stationary. Using that $\II$ is stationarily tall, we can fix a stationary set~$S' \in\II$ such that $S' \subseteq (\kappa \setminus x) \cap S$. Let $y\in [x\restr S']$. Since $S'$ is disjoint from~$x$, $y$ is disjoint from $S'$. Assume there exists a club $B$ such that $y\supseteq B\cap S$; then also $B\cap S\cap S'=\emptyset$, but $S\cap S'=S'$ is stationary, which
  contradicts the fact that $B$ is club. Hence $y\notin \cluf_\kappa^S$, and therefore $\cluf_\kappa^S$ is disjoint from~$[x \restr S']$, as desired.
\end{proof}

Let us finally remark the following, which was brought to our attention by Vincenzo Dimonte.

\begin{observation}
  For any ideal $\II$ on $\kappa$, if
  $2^{(2^{<\kappa})}=2^\kappa$
  (in particular, if $2^{<\kappa}=\kappa$),
  then
  there is a set that
  does not have the $\II$-Baire property, and hence
  is not $\II$-Borel.
\end{observation}

\begin{proof}
  Since $2^{(2^{<\kappa})}=2^\kappa$,
  there exists a Bernstein subset of $2^\kappa$ (in the sense of the higher Cantor space $2^\kappa$), that is a set $X$ such that both $X$ and $2^\kappa\setminus X$ intersect every perfect subset of $2^\kappa$,
   simply because by our assumption, there are only $2^\kappa$-many perfect subsets of $2^\kappa$.
   Let us argue that
   a Bernstein set cannot have the $\II$-Baire property (and hence is not $\II$-Borel).
   If it were $\II$-meager, its complement would contain a perfect set by Corollary \ref{corollary:mycielski}, a contradiction. But otherwise, our Bernstein set would have to be
   $\II$-comeager in an $\II$-cone, but then it would contain a perfect set by a relativized version of Corollary \ref{corollary:mycielski}, which is also a contradiction.
\end{proof}

We do not know how to come up with an example of a non-$\II$-Borel set if $2^{<\kappa}>\kappa$ and $\II$ is stationarily tall:

\begin{question}
Does there always exist a set which is not $\II$-Borel (or even does not have the $\II$-Baire property)?
\end{question}

\section{Sequences in ideal topologies}\label{section:sequences}

\subsection{Convergence and accumulation points}

A prominent notion in analysis is that of \emph{$\II$-convergence}, a generalized notion of convergence with respect to an ideal $\II$ on the set of natural numbers. The idea is that for a sequence to $\II$-converge, it only needs to enter every neighbourhood (in the bounded topology) of its limit on a set in $\II^*$, thus yielding a weakening of the standard notion of convergence. Such a generalized notion of convergence -- namely statistical convergence -- was first considered in \cite{Steinhaus} and~\cite{fast}, and the generalized notion of $\II$-convergence was introduced only much later in \cite{ksw}. If we consider topologies other than the bounded topology, it seems most natural to generalize the concept of convergence in two respects, to that of \emph{$(\II,\JJ)$-convergence}.

\begin{definition}
  Given ideals $\II$ and $\JJ$ on $\kappa$, and a sequence $\vec x=\langle x_\alpha\mid\alpha<\kappa\rangle$ of elements of $2^\kappa$, we say that
  \begin{itemize}
    \item $\vec x$ \emph{$(\II,\JJ)$-converges} to $x\in 2^\kappa$ if for every $\II$-open set $\mathcal O$ containing $x$, $\{\alpha<\kappa\mid x_\alpha\in\mathcal O\}\in\JJ^*$; we call $x$ the \emph{$(\II,\JJ)$-limit} of $\vec x$;
    \item $\vec x$ \emph{$\II^2$-converges} to $x$ if $\vec x$ $(\II,\II)$-converges to $x$; we call $x$ the \emph{$\II^2$-limit} of $\vec x$;\footnote{When context makes this obvious, we may sometimes talk about \emph{limits} when we actually mean $(\II,\JJ)$-limits or $\II^2$-limits.}
    \item as usual, of course, to $(\II,\JJ)$-converge means to $(\II,\JJ)$-converge to some $x\in 2^\kappa$; similarly for $\II^2$-convergence.
  \end{itemize}
\end{definition}

If we only change the ideal that induces our topology,
yet leave the condition for convergence as usual, then in many cases, we do not obtain an interesting notion:

\begin{proposition}
Assume that $\II$ is tall, and that $\vec x=\langle x_i\mid i<\kappa\rangle$ does $(\II,\bd_\kappa)$-converge to $x\in 2^\kappa$. Then, $\vec x$ is eventually constant.
\end{proposition}
\begin{proof}
Assume for a contradiction that $\vec{x}$ is not eventually constant.
Then there are strictly increasing sequences $\vec{\alpha}=\langle \alpha_i\mid i<\kappa\rangle$ and $\vec{\beta}=\langle\beta_i\mid i<\kappa\rangle$ with $x_{\beta_i}(\alpha_i)\neq x(\alpha_i)$ for all $i<\kappa$.
Let $A=\{\alpha_i \mid i<\kappa\}$.
Since $\II$ is tall, there's an unbounded subset $B$ of $A$ with $B\in\II$.
But then, $x_{\beta_i}\notin [x{\upharpoonright}B]$ for all $i\in B$, contradicting the assumption of the proposition.
\end{proof}

Let us provide examples showing that $\bd_\kappa^2$-convergence and $\II^2$-convergence for $\II\supsetneq\bd_\kappa$ are independent of each other:

\begin{observation}\label{observation:compareconvergence}
Let $\II$ be an ideal on $\kappa$ that contains an unbounded subset $A$ of~$\kappa$. Then, the following hold true.
\begin{enumerate}
    \item There are $\II^2$-convergent sequences that are not  $\bd_\kappa^2$-convergent.
    \item There are  $\bd_\kappa^2$-convergent sequences that are not $\II^2$-convergent.
  \end{enumerate}
\end{observation}
\begin{proof}
  \begin{enumerate}
    \item Any sequence such that $x_\alpha=\zero$ for all $\alpha$ in $\kappa\setminus A$ is $\II^2$-convergent with limit $\zero$,
    but it will not be  $\bd_\kappa^2$-convergent for example if we additionally let $x_\alpha=\one$ for all $\alpha$ in $A$.
    \item Let $\vec x=\langle x_\alpha\mid \alpha<\kappa\rangle$ enumerate $A$ in increasing order. Then, $\vec y=\langle\{x_\alpha\}\mid\alpha<\kappa\rangle$ is clearly  $\bd_\kappa^2$-convergent with limit $\zero$, and it is easy to see that $\zero$ is the only possible $\II^2$-limit for $\vec y$,
    yet the cone $[\zero\restr A]$ contains no $\{x_\alpha\}$, yielding that $\vec y$
    does not $\II^2$-converge.\footnote{\label{footnote:also_for_accu}The proof clearly shows that $\vec{y}$ has not even an $\II^2$-accumulation point (see Definition~\ref{definition:accumulation_point}).}\qedhere
  \end{enumerate}
\end{proof}

A notion closely connected to convergence is that of an accumulation point of a sequence. Given the above, it is pretty obvious what should be the right definition of this notion in our generalized context.

\begin{definition}\label{definition:accumulation_point}
  With $\II$, $\JJ$ and $\vec x$ as above, we say that $x\in 2^\kappa$ is an \emph{$(\II,\JJ)$-accumulation point} of $\vec x$ in case that for every $\II$-open set $\mathcal O$ containing $x$, $\{\alpha<\kappa\mid x_\alpha\in\mathcal O\}\in\JJ^+$. \emph{$\II^2$-accumulation points} are $(\II,\II)$-accumulation points.
\end{definition}

We make some simple observations:

\begin{observation}\
  \begin{enumerate}
    \item If $\vec x$ is $(\II,\JJ)$-convergent with limit $x\in 2^\kappa$, then $x$ is the unique $(\II,\JJ)$-accumulation point of $\vec x$ -- this is easily shown as usual (using that $\JJ^*\subseteq\JJ^+$). It is then also
a $\bd_\kappa^2$-accumulation point of $\vec x$ (using that $\JJ^*\subseteq\ub_\kappa$).
    \item
Any $(\II,\JJ)$-convergent sequence has a unique $(\II,\JJ)$-limit, and if $\JJ$ is a maximal ideal, then every sequence has at most one $(\II,\JJ)$-accumulation point. (This uses that the $\II$-topology is Hausdorff and $\JJ^*$ is a filter.)
  \end{enumerate}
\end{observation}

Our next observation is trivial to verify.

\begin{observation}\label{observation:lim_accu}
Let $\II\subseteq \II'$ and $\JJ\subseteq \JJ'$ be ideals on $\kappa$ and let $\vec{x}$ be a sequence. The following implications hold:
\begin{enumerate}
\item $z \textrm { is the } (\II',\JJ)\textrm{-limit of } \vec{x}\Rightarrow
    z \textrm{ is the } (\II,\JJ')\textrm{-limit of }\vec{x}$.
\item $z \textrm{ is an } (\II',\JJ')\textrm{-accumulation point of } \vec{x}\Rightarrow$\\
    $z \textrm{ is an } (\II,\JJ)\textrm{-accumulation point of }\vec{x}$.
\end{enumerate}
\end{observation}

Let us observe the following.

\begin{observation}\label{observation:compact_accu}
Let $\II$ and $\JJ$ be ideals on $\kappa$, such that $\II$ contains an unbounded subset of $\kappa$.
Then, there is a sequence $\vec x$ without an $(\II,\JJ)$-accumulation point.
\end{observation}
\begin{proof}
Using Observation \ref{observation:manyopensets}, let $\{[f_\alpha]\mid \alpha<\kappa\}$ be a family of disjoint $\II$-cones. For every $\alpha<\kappa$, let $x_\alpha\in[f_\alpha]$. Then, $\vec x=\langle x_\alpha\mid\alpha<\kappa\rangle$ is clearly as desired.
\end{proof}

When $\II\supseteq \NS_\kappa$, then there are strong restrictions on what $\II^2$-convergent sequences (and also sequences with $\II^2$-accumulation points) can look like.

\begin{definition}\label{definition:sequentially_tall}
We say that an ideal~$\II$ on $\kappa$ is \emph{sequentially tall} if
for each sequence $\{ y_i \mid i < \kappa \}$ of
unbounded subsets of~$\kappa$,
there exists a set
$y\in \II$
such that $y$ has
non-empty intersection with every $y_i$.
\end{definition}

It is straightforward to check that every sequentially tall ideal is tall.
However, we do not know whether the converse holds true:

\begin{question}
Is there a tall ideal which is not sequentially tall?
\end{question}

It is easy to see (for a similar proof, see Observation~\ref{observation:stationary_tall_etc}(2))
that every maximal ideal is sequentially tall.
Note that if $\II$
is
sequentially tall, then any  $\II' \supseteq \II$ is sequentially tall.

\begin{proposition}\label{proposition:NS_sequentially_tall}
$\NS_\kappa$ is sequentially tall.
\end{proposition}

\begin{proof}
Given $\{ y_i \mid i < \kappa \}$ with $y_i$ unbounded for each $i$,
define $y=\{ \gamma_i \mid i < \kappa\} $ recursively as follows.
Let $\delta_0\in \kappa$ be arbitrary and pick $\gamma_0> \delta_0$ such that $\gamma_0\in y_0 $.
Assume that $\delta_i$ and $\gamma_i$ have been defined. Pick some $\delta_{i+1}>\gamma_i$ and some $\gamma_{i+1}>\delta_{i+1}$, such that $\gamma_{i+1}\in y_{i+1}$.
If $i$ is a limit ordinal, and $\delta_j$ and $\gamma_j$ have been defined for each $j<i$, let $\delta_i := \sup_{j<i} \delta_j$ and pick some $\gamma_i>\delta_i$ such that $\gamma_i\in y_i$.

Note that $\{ \delta_i \mid i <\kappa\}$ is a club subset of $\kappa$ which is disjoint from $y$, hence $y\in \NS_\kappa$, and $\gamma_i\in y\cap y_i$ for each $i < \kappa$.
\end{proof}

\begin{lemma}\label{lemma:convergencerestriction}
  Let $\II$ and $\JJ$ be ideals,
  let $\II$ be sequentially tall
  and let $\vec x=\langle x_\alpha\mid\alpha<\kappa\rangle$ be a sequence that $(\II,\JJ)$-converges to $x\in 2^\kappa$. Then, there is a set $C\in\JJ^*$ such that
  the symmetric difference\footnotemark{} $x_\alpha\Delta x$
  is bounded whenever $\alpha\in C$. If $x$ is only an $(\II,\JJ)$-accumulation point of $\vec x$, then we only get this for a set of $\alpha$'s in $\JJ^+$ (rather than $\JJ^*$).
\end{lemma}

\footnotetext{Note that $x \Delta y = x + y$,
where $x + y$ is the bitwise sum (modulo $2$) of $x$ and~$y$.}

\begin{proof}
 Considering $x_\alpha\Delta x$ rather than $x_\alpha$ for every $\alpha<\kappa$, we may as well assume that $x=\zero$.
Since $\II$ is sequentially tall,  there exists $y \in \II$
which intersects every $x_\alpha$ which is unbounded, hence $[\zero\restr y]$ contains no unbounded $x_\alpha$,
yet it is supposed to contain $\JJ^*$-many (or, $\JJ^+$-many if we consider the case of $x$ being only an $(\II,\JJ)$-accumulation point) $x_\alpha$'s. Thus, $\JJ^*$-many (or, $\JJ^+$-many) $x_\alpha$'s have to be bounded, i.e., after translating back via taking symmetric differences with $x$ once again, $x_\alpha\Delta x$ is bounded, as desired.
\end{proof}

Let us say that for a set $X\subseteq 2^\kappa$, the \emph{$\II$-closure} of $X$ is the $\subseteq$-minimal $\II$-closed set $Y\supseteq X$, which exists
because
$\II$-closed sets are closed under the taking of arbitrary intersections.

\begin{observation}\label{observation:I-closure}
The $\II$-closure of $\II^*$
is all of $2^\kappa$.
\end{observation}
\begin{proof}
Since every non-empty $\II$-open set contains an element of $\II^*$, it follows that the complement of the $\II$-closure of $\II^*$ has to be empty, i.e., that the $\II$-closure of $\II^*$ is all of $2^\kappa$.
\end{proof}

Our next result shows that for many interesting ideals $\II$, $\II$-closure cannot be characterized through
$(\II,\JJ)$-limits or through $(\II,\JJ)$-accumulation points of sequences.

\begin{corollary}
  Let $\II$ and $\JJ$ be ideals, where
$\II$ is sequentially tall.
Then, the following hold true:
  \begin{enumerate}
    \item If $\vec x$ is a sequence of elements of $\II^+$, then all $(\II,\JJ)$-accumulation points
of $\vec x$ are in $\II^+$. In particular, if $\vec{x}$ is $(\II,\JJ)$-convergent, its $(\II,\JJ)$-limit is in~$\II^+$.\footnote{The same holds true for $\II^*$ in place of $\II^+$.}
    \item $\II^+$ is closed under $(\II,\JJ)$-accumulation points of sequences, and under $(\II,\JJ)$-limits of sequences. In particular, this shows that $\II$-closed sets cannot be characterized as being closed under $(\II,\JJ)$-accumulation points or closed under $(\II,\JJ)$-limits.
\end{enumerate}
\end{corollary}
\begin{proof}
  \begin{enumerate}
    \item By Lemma \ref{lemma:convergencerestriction}, if $\vec x$ had an $(\II,\JJ)$-accumulation point outside of~$\II^+$, then, using that $\bd_\kappa\subseteq\II$, $\JJ^+$-many $x_\alpha$'s would have to be outside of~$\II^+$, contradicting our assumption.
    \item Immediate by (1) and by Observation~\ref{observation:I-closure}, for $\II^+\supseteq\II^*$.\qedhere
  \end{enumerate}
\end{proof}

We have so far only seen one trivial example of an $\II^2$-convergent sequence in Observation \ref{observation:compareconvergence}, (1). Let us provide an example of an $\II^2$-convergent sequence which is slightly less trivial, in case $\II=\NS_\kappa$:

\begin{observation}\label{observation:enumerationsofclubsconverge}
  Let $\II=\NS_\kappa$, and let $x\in\II^*$ be enumerated in increasing order by $\langle x_\alpha\mid\alpha<\kappa\rangle$. Then, $\vec y=\langle\{x_\alpha\}\mid\alpha<\kappa\rangle$ $\II^2$-converges to $\zero$.
\end{observation}
\begin{proof}
  Let $A\subseteq\kappa$ be nonstationary. Let $C\subseteq x$ be a club subset of $\kappa$ that is disjoint from $A$, and let $D$ be the set of indices $\alpha$ such that $x_\alpha\in C$. Note that $D$ is again a club subset of $\kappa$. Now, $[\zero\restr A]$ contains $\{x_\alpha\}$ for all $\alpha\in D$, yielding that $\vec y$ does indeed $\II^2$-converge to $\zero$.
\end{proof}

\subsection{Subsequences}

It should not come as a surprise that the usual concept of subsequence is not of much use in generalized $\II$-topologies.\footnote{With $\vec y$ being a \emph{subsequence} of $\vec x$, we mean (as usual) that there is a strictly increasing sequence of ordinals $\langle\beta_\alpha\mid\alpha<\kappa\rangle$, and $y_\alpha=x_{\beta_\alpha}$ for every $\alpha<\kappa$.} Let us demonstrate this with the following trivial observation:

\begin{observation}\label{observation:II_convergent}
  Assume that $\II$ contains an unbounded subset $A$ of $\kappa$. Then, the following hold true:
  \begin{enumerate}
    \item There is a sequence $\vec x$ with no $\II^2$-accumulation points which has an $\II^2$-convergent subsequence $\vec y$.
     \item There is an $\II^2$-convergent sequence (that is also  $\bd_\kappa^2$-convergent) with a subsequence that has no $\II^2$-accumulation points.
     \item There is an $\II^2$-convergent sequence with an $\II^2$-convergent subsequence that has a different $\II^2$-limit.
  \end{enumerate}
\end{observation}
\begin{proof}
  \begin{enumerate}
    \item Simply take $x_\alpha$ to have constant value for an unbounded subset of $\alpha$'s in $\II$, and for the other $\alpha$'s, pick $x_\alpha$ in disjoint $\II$-cones $[f_\alpha]$, as provided in Observation \ref{observation:manyopensets}. Then, $\vec x$ is clearly as desired, with the subsequence of $y_\alpha$'s being of the above constant value.
    \item Let $x_\alpha=\{\alpha\}$ for $\alpha<\kappa$. Then $\vec x=\langle x_\alpha\mid\alpha<\kappa\rangle$ is easily seen to be $\II^2$-convergent,
    yet the sequence of $\{\alpha\}$'s for $\alpha\in A$ is a subsequence of $\vec x$ that does not
have an $\II^2$-accumulation point (see Observation \ref{observation:compareconvergence} (2) and Footnote~\ref{footnote:also_for_accu}).
    \item Easy, and essentially the same as the proof of \cite[Proposition 3.1~(ii)]{ksw}.\qedhere
  \end{enumerate}
\end{proof}

The following easy observation provides a positive result for certain ideals~$\II$:

\begin{observation}
If $\II$ has a
basis
of size $\kappa$, then every sequence with an $(\II,\JJ)$-accumulation point $x$ has a subsequence with $(\II,\JJ)$-limit $x$.
\end{observation}

Let us propose the following generalized notion of subsequence, which corresponds to the usual notion of subsequence being based on an \emph{unbounded} set of indices when working with the \emph{bounded} topology:

\begin{definition}
  Let $\vec x=\langle x_\alpha\mid\alpha<\kappa\rangle$ be a sequence in $2^\kappa$. We say that $\vec y=\langle y_\alpha\mid\alpha<\kappa\rangle$ is a
$\JJ$-subsequence of $\vec x$, and denote this property as $\vec y \Jsubseq  \vec x$, if there is a strictly increasing sequence $\langle\beta_\alpha\mid\alpha<\kappa\rangle$ of ordinals below $\kappa$ such that $y_\alpha=x_{\beta_\alpha}$ for every $\alpha<\kappa$, and such that $\{\beta_\alpha\mid\alpha<\kappa\}\in\JJ^+$.
\end{definition}

The following definition provides two properties of ideals which will be shown to be equivalent to some natural properties of the $\JJ$-subsequence relation. If $S$ is a set of ordinals, let $\pi_S$ denote the transitive collapsing map of $S$.

\begin{definition}
Let $\JJ$ be an ideal on $\kappa$. We say that $\JJ$ is

\begin{itemize}

\item \emph{closed under re-enumerations} if
for each $S \in \JJ^+$ and each $A \subseteq S$ with $A \in \JJ$, we have
$\pi_S[A] \in \JJ$.

\item \emph{almost closed under re-enumerations} if
for each $S \in \JJ^+$ and each $A \subseteq S$ with $A \in \JJ$, we have $\pi_S[A] \notin \JJ^*$.

\end{itemize}
\end{definition}

\begin{observation} \label{observation:max_re-enum}
If $\JJ$ is a maximal ideal which is almost closed under re-enumer\-ations, then it is closed under re-enumerations, just because $\JJ^+=\JJ^*$.
\end{observation}

\begin{proposition}
$\bd_\kappa$ is
closed under re-enumerations.
\end{proposition}
\begin{proof} This holds because being in $\bd_\kappa$ is just a matter of cardinality:
Assume $S$ is in $\bd_\kappa^+$, i.e., it is has size $\kappa$ and $A\subseteq S$ is bounded, i.e., it has size $<\kappa$. Then clearly $\pi_S[A]$ has size $<\kappa$, so $\pi_S[A]\in \bd_\kappa$.
\end{proof}

For
$\lambda < \kappa$ regular, let
$\cof\,\lambda := \{ \alpha < \kappa \with \cof(\alpha ) = \lambda \}$.

\begin{proposition}\label{proposition:NSnotclosedunderreenumerations}
If $\kappa>\omega_1$, then $\NS_\kappa$ is
not closed under re-enumerations.
\end{proposition}

\begin{proof}
Let $S=\cof\,\omega$, and let $A=\{\alpha+\omega\mid\cof(\alpha)=\omega_1\}$. Then, $A$ is clearly nonstationary, as witnessed by the club that is the closure of $\cof\,\omega_1$,
yet $\pi_S[A]=\cof\,\omega_1$ is stationary.
\end{proof}

\begin{question}
Is
 $\NS_{\omega_1}$ closed under re-enumerations?
\end{question}

However, the weaker of the above properties does hold true for $\NS_\kappa$:

\begin{proposition}\label{prop:almost_closed_under}
Let $\JJ \supseteq \NS_\kappa$. Then $\JJ$ is
almost closed under re-enumerations.
\end{proposition}

This implies, together with Observation~\ref{observation:max_re-enum}, that any maximal ideal $\JJ\supseteq \NS_\kappa$ is closed under re-enumerations.

\begin{proof}[Proof of Proposition~\ref{prop:almost_closed_under}]
Let $S \in \JJ^+$ and $A \subseteq S$ with $A \in \JJ$.

Let us first prove the following claim:
    \begin{claim*}
      There is a club $D$ such that $\pi_S\restr\,(D\cap S)=\id$.
    \end{claim*}
    \begin{proof}
		Assume not, i.e., for every club $D$, $\pi_S\restr\,(D\cap S)\neq\id$. Hence the set $S':=\{\alpha\in S \mid \pi_S(\alpha) \neq \alpha\}$ is stationary. Since $\pi_S$ is a transitive collapsing map, $\pi_S(\alpha)<\alpha$ whenever $\pi(\alpha)\neq \alpha$. So $\pi_S\restr S'$ is a regressive function, and hence by Fodor, there is a stationary subset on which it is constant, contradicting the fact that $\pi_S$ is injective.
    \end{proof}

		Let $C\in \JJ^*$ be the complement of $A$.
    Use the claim to obtain $D$, and observe that $D\in \JJ^*$ and let $D':= D\cap C\in \JJ^*$; then $\pi_S[D'\cap S]=D'\cap S\in \JJ^+$ is disjoint from $\pi_S[A]$, showing that $\pi_S[A]\notin \JJ^*$.
\end{proof}

\begin{proposition}\label{proposition:re-enumerations}
Let $\II$ and $\JJ$ be ideals on $\kappa$.
Then, the following are equivalent:
  \begin{enumerate}

    \item $\JJ$ is closed under re-enumerations.

    \item Whenever $\vec y$ is a $\JJ$-subsequence of $\vec x$, then
    \[
    z \textrm{ is the } (\II,\JJ)\textrm{-limit of } \vec x \implies
    z \textrm{ is the } (\II,\JJ)\textrm{-limit of } \vec y.
    \]

    \item Whenever $\vec y$ is a $\JJ$-subsequence of $\vec x$, then
    \[
    z \textrm{ is an } (\II,\JJ)\textrm{-accumulation point of } \vec y \implies  \]
    \[
    z \textrm{ is an } (\II,\JJ)\textrm{-accumulation point of } \vec x.
    \]

    \item Being a $\JJ$-subsequence is a transitive relation, i.e.,
    \[
    \vec z\Jsubseq  \vec y \Jsubseq  \vec x
    \implies
    \vec z \Jsubseq  \vec x.
    \]

  \end{enumerate}
\end{proposition}

\begin{proof}
\emph{(1)$\implies$(2)}: Let $z$ be the $(\II,\JJ)$-limit of $\vec{x}$, so for every $\II$-neighbourhood $\mathcal{O}$ of $z$ there is a $\JJ^*$-set $C$ such that $x_\alpha\in \mathcal{O}$ for every $\alpha\in C$. Let $S\in \JJ^+$ be the index set of $\vec{y}$. Since $\JJ$ is closed under re-enumerations, $\pi_S[C\cap S]\in \JJ^*$, and therefore, $z$ is the $(\II,\JJ)$-limit of $\vec{y}$.

\emph{(1)$\implies$(3)}:  Let $z$ be an $(\II,\JJ)$-accumulation point of $\vec{y}$, so for every $\II$-neighbourhood $\mathcal{O}$ of $z$ there is a $\JJ^+$-set $S'$ such that $y_\alpha\in \mathcal{O}$ for every $\alpha\in S'$. Let $S\in \JJ^+$ be the index set of $\vec{y}$. Since $\JJ$ is closed under re-enumerations,
$\pi_S^{-1}[S']\in \JJ^+$, and therefore, $z$ is an $(\II,\JJ)$-accumulation point of $\vec{x}$.

\emph{(1)$\implies$(4)}:
Let $S\in \JJ^+$ be the index set of $\vec{y}$ as a subsequence of $\vec{x}$, and let $S'\in\JJ^+$ be the index set of $\vec{z}$ as a subsequence of $\vec{y}$. Since $\JJ$ is closed under re-enumerations,
we have $\pi_S^{-1}[S']\in \JJ^+$,
hence $\vec{z}$ is a $\JJ$-subsequence of $\vec{x}$.

\emph{(2)$\implies$(1)}: Assume that $\JJ$ is not closed under re-enumerations, as witnessed by $S$ and by $A$. Let $x_\alpha=\one$ for $\alpha\in A$, and let $x_\alpha=\zero$ for $\alpha\notin A$. Clearly, $\zero$ is the $(\II,\JJ)$-limit of $\vec{x}$. Now $\vec{y}:=\{ x_\alpha\mid \alpha\in S\}$ is a $\JJ$-subsequence, and $\zero$ is not the $(\II,\JJ)$-limit of $\vec{y}$, because $\pi_S[A]\in \JJ^+$.

\emph{(3)$\implies$(1)}: Assume that $\JJ$ is not closed under re-enumerations, as witnessed by $S$ and by $A$. Let $x_\alpha=\one$ for $\alpha\in A$ and $x_\alpha=\zero$ for $\alpha\notin A$. Clearly $\one$ is not an $(\II,\JJ)$-accumulation point of $\vec{x}$. Now $\vec{y}:=\{ x_\alpha\mid \alpha\in S\}$ is a $\JJ$-subsequence, and $\one$ is an $(\II,\JJ)$-accumulation point of $\vec{y}$, because $\pi_S[A]\in \JJ^+$.

\emph{(4)$\implies$(1)}: Assume that $\JJ$ is not closed under re-enumerations, as witnessed by $S$ and by $A$. Let $\vec{x}$ be a sequence, let $\vec{y}:=\{ x_\alpha\mid \alpha\in S\}$, and let $\vec{z}:=\{ x_\alpha\mid \alpha\in A\}$. Then, $\vec{z}\Jsubseq  \vec{y}\Jsubseq \vec{x}$, 
yet $\vec{z}\Jsubseq \vec{x}$ does not hold.
\end{proof}

It follows that $\bd_\kappa$-subsequences have all these nice properties, while $\NS_\kappa$-sub\-se\-quen\-ces do not have them (at least if $\kappa>\omega_1$). But since $\NS_\kappa$ is almost closed under re-enumerations, the following proposition gives weaker properties which hold for $\NS_\kappa$-subsequences.

\begin{proposition}\label{proposition:almost_re-enumerations}
Let $\II$ and $\JJ$ be ideals on $\kappa$.
Then, the following are equivalent:
  \begin{enumerate}

    \item $\JJ$ almost closed under re-enumerations.

    \item Whenever $\vec y$ is a $\JJ$-subsequence of $\vec x$, then
    \[
    z \textrm{ is } (\II,\JJ)\textrm{-limit of } \vec x \implies
    z \textrm{ is } (\II,\JJ)\textrm{-accumulation point of } \vec y.
    \]

    \item Whenever $\vec y$ is a $\JJ$-subsequence of $\vec x$, then
    \[
    z \textrm{ is } (\II,\JJ)\textrm{-limit of } \vec y \implies
    z \textrm{ is } (\II,\JJ)\textrm{-accumulation point of } \vec x.
    \]

  \end{enumerate}
\end{proposition}

\begin{proof}
\emph{(1)$\implies$(2)}:
Let $\mathcal O$ be an $\II$-open set containing $x$. Since $\vec x$ $(\II,\JJ)$-converges to $x$, $x_\alpha$ is an element of $\mathcal O$ for a $\JJ^*$-set of $\alpha$'s. But then, since $\JJ$ is almost closed under re-enumerations this implies that $y_\alpha$ is an element of $\mathcal O$ for a $\JJ^+$-set of $\alpha$'s, yielding that $x$ is an $(\II,\JJ)$-accumulation point of $\vec y$, as desired.

\emph{(1)$\implies$(3)}:  Let $z$ be an $(\II,\JJ)$-limit of $\vec{y}$, so for every $\II$-neighbourhood $\mathcal{O}$ of $z$ there is a $\JJ^*$-set $C$ such that $y_\alpha\in \mathcal{O}$ for every $\alpha\in C$. Let $S\in \JJ^+$ be the index set of $\vec{y}$. Since $\JJ$ is closed under re-enumerations, $\pi_S^{-1}[C]\in \JJ^+$, and therefore $z$ is an $(\II,\JJ)$-accumulation point of $\vec{x}$.

\emph{(2)$\implies$(1)}: Assume that $\JJ$ is not almost closed under re-enumerations and let $S$ and $A$ be witnesses for that. Let $x_\alpha=\one$ for $\alpha\in A$ and $x_\alpha=\zero$ for $\alpha\notin A$. Clearly, $\zero$ is the $(\II,\JJ)$-limit of $\vec{x}$. Now $\vec{y}:=\{ x_\alpha\mid \alpha\in S\}$ is a $\JJ$-subsequence, and $\zero$ is not an $(\II,\JJ)$-accumulation point of $\vec{y}$, because $\pi_S[A]\in \JJ^*$. In fact, $\one$ is the $(\II,\JJ)$-limit of $\vec{y}$.

\emph{(3)$\implies$(1)}: Assume that $\JJ$ is not almost closed under re-enumerations and let $S$ and $A$ be witnesses for that. Let $x_\alpha=\one$ for $\alpha\in A$ and $x_\alpha=\zero$ for $\alpha\notin A$. Clearly $\one$ is not an $(\II,\JJ)$-accumulation point of $\vec{x}$. Now $\vec{y}:=\{ x_\alpha\mid \alpha\in S\}$ is a $\JJ$-subsequence of $\vec x$, and $\one$ is the $(\II,\JJ)$-limit of $\vec{y}$, because $\pi_S[A]\in \JJ^*$.
\end{proof}

\begin{remark}\label{remark:not_unique_acc}
If $\JJ$ is almost closed under re-enumerations, but not closed under re-enumerations, then the example from the proof of \emph{(2)$\implies$(1)} of Proposition~\ref{proposition:re-enumerations} yields sequences $\vec y \Jsubseq  \vec x$ such that $x$ is the $(\II,\JJ)$-limit of $\vec x$, such that $x$ is an $(\II,\JJ)$-accumulation point of $\vec{y}$ (by (2) of Proposition~\ref{proposition:almost_re-enumerations}), and such that $x$ is not the unique $(\II,\JJ)$-accumulation point of $\vec{y}$.
\end{remark}

Let us provide an example of a strong failure of Proposition~\ref{proposition:re-enumerations}.

\begin{proposition}\label{proposition:subsequencenegative}
Let $\II$ be an ideal on $\kappa$ that contains an unbounded subset of $\kappa$, and let $\JJ$ be an ideal that is not closed under re-enumerations. Then, there is a sequence $\vec x$ with no $(\II,\JJ)$-accumulation points, which has a $\JJ$-subsequence $\vec{y}$ with a unique $(\II,\JJ)$-accu\-mu\-lation point, such that $\vec{y}$ has an $(\II,\JJ)$-convergent $\JJ$-subsequence~$\vec{z}$.\footnote{Note that, since $\NS_\kappa$ is almost closed under re-enumerations, this $\JJ$-subsequence $\vec{y}$ cannot be $(\II,\JJ)$-convergent, for otherwise its $(\II,\JJ)$-limit were an $(\II,\JJ)$-accumulation point of $\vec{x}$.
 This also shows again that being a $\JJ$-subsequence is not a transitive relation, because if $\vec{z}$ were a $\JJ$-subsequence of $\vec{x}$, its $(\II,\JJ)$-limit would be an $(\II,\JJ)$-accumulation point of $\vec{x}$ by Proposition~\ref{proposition:almost_re-enumerations}.}
\end{proposition}
\begin{proof}
Let $S$ and $A$ be witnesses for $\JJ$ not being closed under re-enumerations. By Observation \ref{observation:manyopensets}, we may pick disjoint $\II$-cones $[f_\alpha]$ for $\alpha<\kappa$. Define a sequence $\vec x$ by taking $x_\alpha$ to have constant value on $A$, and for $\alpha\not\in A$, pick some $x_\alpha\in[f_\alpha]$. Then, $\vec x$ has no $(\II,\JJ)$-accumulation points. However, letting $\langle\beta_\alpha\mid\alpha<\kappa\rangle$ enumerate $S$ in increasing order, and letting $y_\alpha=x_{\beta_\alpha}$ for $\alpha<\kappa$, we obtain a sequence $\vec y$ with unique $(\II,\JJ)$-accumulation point, which is the constant value on the $\JJ^+$-set of indices $\pi_S[A]$. The final statement of the proposition now follows by considering the $\JJ$-subsequence of $\vec y$ with domain $\pi_S[A]$.
\end{proof}

We give
yet another example of a sequence that does not $\II^2$-converge, for $\II=\NS_\kappa$. This should be compared to Observation \ref{observation:compareconvergence}, (1) and to Observation \ref{observation:enumerationsofclubsconverge}.

\begin{corollary}
  If $\kappa>\omega_1$ is regular and $\II=\NS_\kappa$, then if $\langle s_\alpha\mid\alpha<\kappa\rangle$ is the increasing enumeration of $\cof\,\omega$, the sequence $\vec s=\langle\{s_\alpha\}\mid\alpha<\kappa\rangle$ does not $\II^2$-converge.
\end{corollary}
\begin{proof}
  The sequence $\vec s$ could only $\II^2$-converge to $\zero$. Let $\pi$ denote the transitive collapsing map of $\cof\,\omega$, and let $A$ be a nonstationary subset of $\cof\,\omega$ for which $\pi[A]$ is stationary, as provided by Proposition \ref{proposition:NSnotclosedunderreenumerations}. This means that there is a stationary set $T$ of indices $\alpha$ for which $s_\alpha\in A$. Assuming that $\vec s$ does indeed $\II^2$-converge to $\zero$, there is a club $C$ of indices $\alpha$ for which $\{s_\alpha\}\in[\zero\restr A]$, i.e., $s_\alpha\not\in A$. Since $T\cap C\ne\emptyset$, 
  this yields a contradiction.
\end{proof}

\section{Connections with topologies generated by forcing partial orders}\label{section:connection}

Ideal topologies can be seen as a special case of topologies connected to \emph{tree-like} forcing notions, that we will describe below, following \cite{fkk}.\footnote{Even more generally, this may be seen as a special case of the natural topology that can be constructed on the Stone space of any partial order, see for example \cite[Section 3]{daisuke}.} While this connection may be interesting enough to be mentioned here in its own right, we will also make use of this connection later on. We start with the following definition from \cite{fkk}, slightly adapted to our present purposes.

\begin{definition}\cite[Definition 3.1]{fkk}
  A forcing notion $\PP$ is called \emph{$\kappa$-tree-like} if
  \begin{enumerate}
    \item conditions of $\PP$ are pruned and ${<}\kappa$-closed trees on $2^{<\kappa}$, ordered by inclusion,    \item $2^{<\kappa}\in\PP$, and whenever $T\in\PP$ and $s\in T$, then $\{t\in T\mid s\subseteq t$ or $t\subseteq s\}\in\PP$, and
    \item if $\langle T_\alpha\mid\alpha<\lambda\rangle$ is a decreasing sequence of conditions of length $\lambda<\kappa$, then $\bigcap_{\alpha<\lambda}T_\alpha\in\PP$.
  \end{enumerate}
\end{definition}

In many cases, $\kappa$-tree-like forcing notions induce natural topologies on $2^\kappa$. The following is a minor modification of \cite[Definition 3.6, 1]{fkk}, as the definition used in that paper seems to be slightly too weak in order to yield a topology basis.

\begin{definition}
  A $\kappa$-tree-like notion of forcing $\PP$ is \emph{topological} if $\{[T]\mid T\in\PP\}$ forms a topology basis for $2^\kappa$, that is whenever $S,T\in\PP$, and $x\in[S]\cap[T]$, then there is $R\in\PP$ such that $x\in[R]\subseteq[S]\cap[T]$. In this case we call the topology generated by the basic open sets of the form $[T]$ for $T\in\PP$ the \emph{topology generated by $\PP$}, or the \emph{$\PP$-topology}.
\end{definition}

Ideal topologies are generated by generalizations of Grigorieff forcing to uncountable cardinals:

\begin{definition}
  Let $\kappa$ be an infinite cardinal and let $\II$ be an ideal on $\kappa$. \emph{Grigorieff forcing with the ideal $\II$} is the notion of forcing consisting of conditions from $\heart_\II$, ordered by reverse inclusion.
\end{definition}

At first sight, Grigorieff forcing may not seem to be a $\kappa$-tree-like notion of forcing, 
but it can be represented as one: We identify a condition $f\in\heart_\II$ with a tree $T$ on $2^{<\kappa}$, which we construct by induction on $\alpha$ as follows: Given $t\in T$ of order-type $\alpha$, let $t^\conc 0\in T$ if and only if $f(\alpha)\ne 1$, and let $t^\conc 1\in T$ if and only if $f(\alpha)\ne 0$ (these are both supposed to include the cases when $\alpha$ is not in the domain of $f$, i.e., $t$ is splitting if and only if $\alpha$ is not in the domain of $f$). At limit levels $\alpha$, we extend every branch through the tree constructed so far. It is straightforward to check that the resulting forcing is indeed a $\kappa$-tree-like forcing, using that $\II\supseteq \bd_\kappa$ and that $\II$ is ${<}\kappa$-complete.
 Note moreover that if $T$ is the tree on $2^{<\kappa}$ corresponding to the condition $f\in\heart_\II$, then $[T]=[f]$.

\medskip

For any ideal $\II$, Grigorieff forcing with the ideal $\II$ is topological: Given $f,g\in\heart_\II$, assuming that $[f]\cap[g]\ne\emptyset$, it follows that $f\cup g\in\heart_\II$, and that $[f\cup g]=[f]\cap[g]$.

\medskip

The following is now immediate by comparing the basic open sets (which are the same) used to generate the respective topologies:

\begin{observation}
  The topology on $2^\kappa$ generated by Grigorieff forcing with the ideal~$\II$ is exactly the topology $\tau_\II$.
\end{observation}

Another notion of forcing that is closely connected to $\II$-topologies is \emph{$\kappa$-Silver forcing}, which is sometimes also called \emph{$\kappa$-club-Silver forcing}:

\begin{definition}
  Given a regular cardinal $\kappa$, $\kappa$-Silver forcing is the notion of forcing consisting of all conditions $f\in\heart_\II$ for which $\kappa\setminus\dom(f)$ is a club
  subset of $\kappa$, ordered by reverse inclusion.
\end{definition}

Note that clearly, $\kappa$-Silver forcing is a dense subset of Grigorieff forcing with the ideal $\NS_\kappa$. In fact, whenever $p$ is a condition in the latter forcing and $x\in 2^\kappa$ is such that $p\subseteq x$, then $p$ can be extended to a condition $q\subseteq x$ in $\kappa$-Silver forcing. This implies both that $\kappa$-Silver forcing can be represented as a $\kappa$-tree-like notion of forcing (see also \cite[Example 3.2, 6]{fkk}), that $\kappa$-Silver forcing is topological,\footnote{The opposite is wrongly claimed in \cite{fkk} without justification, see for example \cite[Table 1]{fkk}.} and that it generates the same topology as does Grigorieff forcing with $\NS_\kappa$, namely the Edinburgh topology on $2^\kappa$.

\medskip

Let us finally remark that it is straightforward to formulate and verify an analogue of Proposition~\ref{normal form} for the more general setting of topologies that are generated by $\kappa$-tree-like forcing notions. We will
leave this to the interested reader, for we will not need it in the remainder of our paper.

\section{On $\kappa$-Silver forcing and Axiom $A^*$}\label{section:axioma}

In \cite[Definition 3.6, 2]{fkk}, a strengthening of Axiom A is defined as follows:

\begin{definition}\label{definition:axioma}
  We say that a $\kappa$-tree-like notion of forcing $\PP$ satisfies Axiom $A^*$ if there are orderings $\{\le_\alpha\mid\alpha<\kappa\}$ with $\le_0=\le$,
  satisfying:\footnote{Since this definition is about $\kappa$-tree-like notions of forcing, one would usually use $S$ and $T$ to denote the elements of the forcings. We use $f$ and $g$ instead, for we will only be concerned with the special case of Grigorieff forcing.}
  \begin{enumerate}
    \item $g\le_\beta f$ implies $g\le_\alpha f$ for all $\alpha\le\beta$.
    \item If $\langle f_\alpha\mid\alpha<\lambda\rangle$ is a sequence of conditions with $\lambda\le\kappa$ satisfying that $f_\beta\le_\alpha f_\alpha$ for all $\alpha\le\beta$, then there is $f\in\PP$ such that $f\le_\alpha f_\alpha$ for all $\alpha<\lambda$.
    \item For all $f\in\PP$, $D$ dense below $f$ in $\PP$, and $\alpha<\kappa$, there exists $E\subseteq D$ and $g\le_\alpha f$ such that $|E|\le\kappa$ and $E$ is predense below $g$, such that\footnotemark{} additionally $[g]\subseteq\bigcup\{[h]\mid h\in E\}$. \footnotetext{Without this additional clause, this would be the usual Axiom $A$.}
  \end{enumerate}
\end{definition}

Many of our subsequent proofs are going to work either under the assumption that $\kappa$ is inaccessible, or that $\diamondsuit_\kappa$ holds.
Recall that, for $\kappa$ being regular uncountable,
$\diamondsuit_\kappa$ holds if there exists a \emph{$\diamondsuit_\kappa$-sequence} $\langle A_\alpha\mid\alpha<\kappa\rangle$, that is, for every $A \subseteq \kappa$, there is a stationary set of $\alpha$'s such that $A_\alpha=A \cap \alpha$.
\begin{definition} We say that $\kappa$ is \emph{simple}, if $\kappa$ is inaccessible or $\diamondsuit_\kappa$ holds.
\end{definition}

The assumption $2^{<\kappa}=\kappa$ is very common in higher descriptive set theory. It is necessary (but not quite sufficient) for $\kappa$ being simple.

\begin{observation}\label{obs:simple_implies} If $\kappa$ is simple, then $2^{{<}\kappa}=\kappa$.
\end{observation}

On the other hand, by results of Shelah (see \cite{shelah}), if $\kappa$ is a successor cardinal and $\kappa>\omega_1$, then $2^{{<}\kappa}=\kappa$ implies that $\diamondsuit_\kappa$ holds. For $\kappa=\omega_1$, however, it is consistent that $2^{<\kappa}=\kappa$ (i.e., CH holds) and $\diamondsuit_{\kappa}$ (i.e., $\diamondsuit$) fails.
The case $\kappa=\omega_1$ is therefore particularly interesting for potential counterexamples.

\medskip

It is well-known that the standard proof to verify that Silver forcing (on $\omega$) satisfies Axiom $A$ can be adapted to show that $\kappa$-Silver forcing satisfies Axiom $A$ in case $\kappa$ is inaccessible, and it is easy to see from the proof that in fact it even yields Axiom $A^*$. We want to show that the same conclusion also holds under the assumption $\diamondsuit_\kappa$. We will then apply this result in the next section.

\begin{theorem}\label{theorem:axioma}
  If $\diamondsuit_\kappa$ holds, then $\kappa$-Silver forcing satisfies Axiom $A^*$. (So $\kappa$-Silver forcing satisfies Axiom $A^*$ whenever $\kappa$ is simple.)
\end{theorem}
\begin{proof}
  Let $\langle\PP,\le\rangle$ denote $\kappa$-Silver forcing. For any $\alpha<\kappa$, let $g\le_\alpha f$ if $g\le f$ and the first $\alpha$-many elements of the complements of the domains of $f$ and of $g$ are the same. It is clear that Items (1) and (2) in Definition \ref{definition:axioma} are thus satisfied, and we only have to verify Item (3).

(If $\kappa$ is inaccessible, this follows from standard arguments, as mentioned above; also compare with the proofs of Theorem~\ref{theorem:cones} and Lemma~\ref{lemma:reaping} in which we provide details for both the case that $\kappa$ is inaccessible and the case that $\diamondsuit_\kappa$ holds.)

 Fix a $\diamondsuit_\kappa$-sequence $\langle A_i\mid i<\kappa\rangle$.
Let $f\in\PP$, let $\alpha<\kappa$, and let $D\subseteq\PP$ be dense below $f$. We inductively construct a decreasing sequence $\langle f_i\mid i\le\kappa\rangle$ of conditions in $\PP$ with $f_i=f$ for $i\le\alpha$, and a sequence $\langle\alpha_i\mid i<\kappa\rangle$ of ordinals with the property that $\langle\alpha_j\mid j\le i\rangle$ enumerates the first $(i+1)$-many elements of $\kappa\setminus\dom(f_i)$ for every $i\le\kappa$, as follows. Let $\langle\alpha_i\mid i\le\alpha\rangle$ enumerate the first $\alpha+1$-many elements of the complement of the domain of $f$.

Assume that we have constructed $f_i$, and also $\alpha_j$ for $j\le i$.

Using that $D$ is dense below $f$, let $g_i^0\le f_i$ be such that
\begin{itemize}
  \item $g_i^0(\alpha_j)=A_i(j)$ for all $j<i$,
  \item $g_i^0(\alpha_i)=0$, and
  \item $g_i^0\in D$,
\end{itemize}
and let $g_i^1\le g_i^0\restr(\dom(g_i^0)\setminus\{\alpha_i\})$ be such that
\begin{itemize}
  \item $g_i^1(\alpha_i)=1$, and
  \item $g_i^1\in D$.
\end{itemize}
Let $f_{i+1}=g_i^1\restr(\dom(g_i^1)\setminus\{\alpha_j\mid j\le i\})$, and note that $f_{i+1}\le_i f_i$, for $\{\alpha_j\mid j<i\}$ is contained in the complement of $\dom(f_{i+1})$. Let $\alpha_{i+1}$ be the least element of $\kappa\setminus\dom(f_{i+1})$ above $\alpha_i$.

For limit ordinals $i\le\kappa$, let $f_i=\bigcup_{j<i}f_j$, and if $i<\kappa$, let $\alpha_i=\bigcup_{j<i}\alpha_j$ be the least element of $\kappa\setminus\dom(f_i)$, using that the intersection of ${<}\kappa$-many club subsets of $\kappa$ is again a club subset of $\kappa$. Let $E=\{g_i^0\mid i<\kappa\}\cup\{g_i^1\mid i<\kappa\}$.

In order to verify Axiom $A$, first note that $f_\kappa\leq_\alpha f$. Now we want to show that $E$ is predense below $f_\kappa$. Thus, let $h\le f_\kappa$ be given. Using the properties of our diamond sequence, pick $i<\kappa$ such that $i\ge\alpha$, and such that for all $j<i$ with $\alpha_j\in\dom(h)$, $A_i(j)=h(\alpha_j)$. Pick $\delta\in\{0,1\}$ such that $h(\alpha_i)=\delta$ in case $\alpha_i\in\dom(h)$. Then, $g_i^\delta$ is compatible to $h$, as desired.

In order to check the additional property for Axiom $A^*$, note that any extension $x$ of $f_\kappa$ to a total function from $\kappa$ to $2$ can be treated in the same way as $h$ above, yielding some $i<\kappa$ and $\delta\in\{0,1\}$ such that $x\in[g_i^\delta]$.
\end{proof}

\section{Edinburgh cones and $\II$-meagerness}\label{section:premeager}

 In this section, we provide two results on properties of simple cardinals with respect to Edinburgh cones, and two consequences on Edinburgh meager sets.

\medskip

The first result is due to S.\ Friedman, Khomskii and Kulikov (\cite[Section~3.2]{fkk}) in case $\kappa$ is inaccessible,
yet, making use of Theorem \ref{theorem:axioma}, it
holds
also under the assumption of $\diamondsuit_\kappa$:

\begin{theorem}\label{theorem:cones2}
  If $\kappa$ is simple, then every $\kappa$-intersection of Edinburgh open dense subsets of $2^\kappa$ contains an Edinburgh open dense set,\footnote{In order to avoid any possible confusion, let us remark that by a convention made earlier on in this paper, \emph{Edinburgh open dense} means Edinburgh open and Edinburgh dense.}
  i.e., the $\edinburgh$ meager sets are the same as the $\edinburgh$ nowhere dense sets.
\end{theorem}
\begin{proof}
  Let us first observe that the two conclusions of the theorem are simply equivalent, using that the Edinburgh closure of an $\edinburgh$ nowhere dense set (that is, the smallest Edinburgh closed set containing it) is still $\edinburgh$ nowhere dense.\footnote{We mostly provide the formulation via Edinburgh open dense sets here as well in order to relate this result more obviously to the statement of Theorem \ref{theorem:cones} below.}
It follows from
Theorem \ref{theorem:axioma}, that $\kappa$-Silver forcing satisfies Axiom $A^*$, and we have observed  in Section~\ref{section:connection} that $\kappa$-Silver forcing is topological and generates the Edinburgh topology on $2^\kappa$. It is straightforward to derive from Axiom $A^*$ that the $\edinburgh$ meager sets are the same as the $\edinburgh$ nowhere dense sets, as desired (see also \cite[Lemma 3.8]{fkk}).\footnote{It would be possible, 
yet less informative, to verify Theorem \ref{theorem:cones2} without making use of Axiom $A^*$, by extending on and making appropriate modifications to the proof of Theorem \ref{theorem:cones} below.}
\end{proof}

Note that the above theorem is a strengthening (of a special case) of Corollary~\ref{corollary:mycielski}.
However, we do not know the answer to the following question. 

\begin{question}\label{question:meagernowheredense}
Is it consistent that the Edinburgh meager sets are not the same as the Edinburgh nowhere dense sets? (By the above theorem, $\kappa$ would necessarily not be simple.)
\end{question}

Note that the
 proof of Theorem~\ref{theorem:cones2} actually yields the following more general fact:
 If Grigorieff forcing with $\II$ satisfies Axiom $A^*$, then $\II$-meager is the same as $\II$-nowhere dense.

\begin{question}\label{question:other_ideal_Axiom_A}
Are there ideals $\II$ other than $\NS_\kappa$ such that Grigorieff forcing with $\II$ satisfies Axiom $A^*$ (or such that $\II$-meager implies $\II$-nowhere dense)?
\end{question}

The arguments for the next result are essentially the same as the argument in the proof of \cite[Lemma 4.9, 6]{fkk} in case $\kappa$ is inaccessible, and they can again be generalized to include the case when $\diamondsuit_\kappa$ holds. Since the context and notation in \cite{fkk} are somewhat different to ours, we would like to include a proof of this theorem in both cases, for the benefit of our readers.

\begin{theorem}\label{theorem:cones}
Let $\kappa$ be simple.
If $s\in\heart_{\bd_\kappa}$,
then every $\kappa$-intersection of open dense subsets of $2^\kappa$ contains an $\edinburgh$ cone $[f]$ with $[f]\subseteq [s]$.
In other words,
every
comeager subset of $2^\kappa$ contains a dense set that is Edinburgh open.
\end{theorem}

\begin{proof}
If $\kappa$ is not inaccessible, $\diamondsuit_\kappa$ holds by the assumption that $\kappa$ is simple; in this case, let us fix a $\diamondsuit_\kappa$-sequence $\langle A_\alpha\mid\alpha<\kappa\rangle$.

\medskip
Let $s\in\heart_{\bd_\kappa}$, and let $\langle D_i\mid i<\kappa\rangle$ be a sequence of open dense subsets of $2^\kappa$. Since the intersection of less than $\kappa$-many open dense subsets of $2^\kappa$ is open dense, we can assume that $D_j\subseteq D_i$ for $j>i$. By induction on $i<\kappa$, we define a $\subseteq$-increasing sequence of functions $\langle f_i\mid i<\kappa\rangle$ in $\heart_{\bd_\kappa}$ with $f_0=s$,
 as well as a club subset $C=\{\alpha_j \with j < \kappa\}$ of $\kappa$ that is disjoint from $\dom(f_i)$ for each $i<\kappa$. Then, letting $f:=\bigcup_{i<\kappa}f_i$, we show that $[f]\subseteq\bigcap_{i<\kappa}D_i$, and since $[f]\subseteq [s]$, this finishes the proof.

 Let $f_0=s$, and pick $\alpha_0>\sup\dom(s)$.
Let $i<\kappa$, and assume that $\langle\alpha_j\mid j\le i\rangle$ and $f_i\in\heart_{\bd_\kappa}$ are defined, such that $\{\alpha_j\mid j\le i\}$ is a closed subset of $\kappa$ that is disjoint from $\dom(f_i)$. For the successor step we consider two cases:

\medskip

\textbf{Case 1: $\kappa$ is inaccessible --}
Let $\prec$ be a wellorder of $2^{i+1}$.
 For $t\in 2^{i+1}$, by induction on $\prec$,
we pick $g_{i}^t\in\heart_{\bd_\kappa}$ such that
\begin{enumerate}
  \item $g_i^t$ extends $f_i$,
  \item $g_i^t(\alpha_j)=t(j)$ for $j\le i$,
  \item $g_i^t$ extends $g_i^u$ on $\dom(g_i^u)\setminus \{ \alpha_j\with j\leq i\}$ whenever $u\in 2^{i+1}$ and $u\prec t$, and
  \item $[g_i^t]\subseteq D_i$.
\end{enumerate}
This is possible because $\kappa$ is inaccessible, hence in particular $|2^{i+1}|<\kappa$, and because $D_i$ is open dense.
Let
\[
f_{i+1}:= \bigcup_{t\in 2^{i+1}} (g_i^t\restr\,(\dom(g_i^{t})\setminus\{\alpha_j\mid j\le i\})),
\]
 and let
 $\alpha_{i+1}>\alpha_i$ such that $\alpha_{i+1}>\sup\dom(f_{i+ 1})$. Note that our above construction ensures that $f_{i+1}\in\heart_{\bd_\kappa}$, and that $[f_{i+1}]\subseteq D_i$, since $[f_{i+1}]\subseteq\bigcup_{t\in 2^{i+1}}[g_i^t]\subseteq D_i$.

\medskip

\textbf{Case 2: $\diamondsuit_\kappa$ holds --}
Using that $D_i$ is open dense, pick $h_i^0\in\heart_{\bd_\kappa}$ such that
\begin{enumerate}
\item $h_i^0$ extends $f_i$,
\item $h_i^0(\alpha_j)=A_i(j)$ for $j<i$,
\item $h_i^0(\alpha_i)=0$, and
\item $[h_i^0] \subseteq D_i$.
\end{enumerate}
Now, using again that $D_i$ is open dense, pick $h_i^1\in\heart_{\bd_\kappa}$ such that
\begin{enumerate}
\item $h_i^1$ extends $h_i^0$ on $\dom(h_i^0)\setminus\{\alpha_i\}$,
\item $h_i^1(\alpha_i)=1$, and
\item $[h_i^1]\subseteq D_i$.
\end{enumerate}
Let $f_{i+1}=h_i^1\,\restr\,(\dom(h_i^1)\setminus\{\alpha_j\mid j\le i\})$, and pick some $\alpha_{i+1}>\sup\dom(f_{i+1})$.

\medskip

 Now again in both cases, for limit ordinals $i\le\kappa$, let $f_i=\bigcup_{j<i}f_j$ and, in case $i<\kappa$, let  $\alpha_i=\sup_{j<i}\alpha_j<\kappa$, by the regularity of $\kappa$. Thus, $C=\{\alpha_i\mid i<\kappa\}$ is a club subset of $\kappa$, and, letting $f=f_\kappa$, $\dom(f)\subseteq\kappa\setminus C$ is nonstationary, yielding that $f\in\heart_{\NS_\kappa}$.

 Our above construction ensured that $[f]\subseteq\bigcap_{i<\kappa}D_i$:
In the case of $\kappa$ inaccessible, this holds by construction, as mentioned above.
Now assume
that $\kappa$ is a successor cardinal for which $\diamondsuit_\kappa$ holds.
Given $x \in [f]$, let $A=\{i<\kappa \with x(\alpha_i)=1\}$.
Since $\langle A_\alpha \mid \alpha<\kappa\rangle$ is a $\diamondsuit_\kappa$-sequence there exists a stationary set $S$ such that
$A \cap i=A_i$ for $i \in S$. For any such $i$, by our above construction of $f_{i+1}$, we have
$x \in [h_i^0]\subseteq D_i$ if $x(\alpha_i)=0$ and $x \in [h_i^1]\subseteq D_i$ if $x(\alpha_i)=1$.
In both cases, $x\in D_i$.
Since we assumed the sequence $\langle D_i \mid i<\kappa\rangle$ to be $\subseteq$-decreasing, and $S$ is unbounded, this yields that $x$ is in the intersection of all the $D_i$, as desired.\qedhere
\end{proof}

The above now allows us to easily infer the following:

\begin{corollary}\label{corollary:imeagerbairemeager}
If $\kappa$ is simple, $X\subseteq 2^\kappa$ has the Baire property and is $\II$-meager for $\II\supseteq\NS_\kappa$, then $X$ is meager.
\end{corollary}
\begin{proof}
Assume that $X$ has the Baire property and is not meager. We show that $X$ is not $\II$-meager. By our assumptions on $X$, there is an $s\in\heart_{\bd_\kappa}$ such that $X\cap [s]$ is comeager in $[s]$. In other words: there exists a sequence $\langle D_\alpha\mid\alpha<\kappa\rangle$ of open dense sets such that $\bigcap_{\alpha<\kappa}D_\alpha\cap [s]\subseteq X$. Applying Theorem \ref{theorem:cones}, there exists $f\supseteq s$, $f\in \heart_{\NS_\kappa}\subseteq\heart_\II$  with $[f]\subseteq \bigcap_{\alpha<\kappa}D_\alpha\cap [s]\subseteq X$, but $[f]$ is not $\II$-meager by Baire Category for the $\II$-topology (see Proposition~\ref{the:BCT}), thus $X$ is not $\II$-meager.
\end{proof}

Again, we do not know the following:
\begin{question}
  Is it consistent that there is $X\subseteq 2^{\kappa}$ with the Baire property which is $\edinburgh$ meager, but not meager?
In particular: Is the above consistent for $\kappa=\omega_1$ together with $2^{<\kappa}=\kappa$
(by Corollary~\ref{corollary:imeagerbairemeager}, $\diamondsuit$ has to fail)?
\end{question}

We will also apply Theorem \ref{theorem:cones} once again in Section \ref{section:baire} below.

\section{The reaping number and some of its variants}\label{section:reaping}

In this section, we will take a small detour in order to investigate some cardinal invariants of the higher Cantor space $2^\kappa$ (with our results
also applying to the classical Cantor space $2^\omega$), and we will apply these results in our later sections. We thus assume that $\kappa$ is a regular infinite cardinal. In particular, we also allow for $\kappa=\omega$. Remember that for $a,b \in\ub_\kappa$, $a$ \emph{splits} $b$ if $a \cap b$ and $b \setminus a$ are both of size $\kappa$.

\begin{definition}\
\begin{itemize}
  \item An \emph{unsplit family} at $\kappa$ is a set $F\subseteq\ub_\kappa$ for which there is no $a\subseteq\kappa$ such that for all $b\in F$, $a$ splits $b$.
  \item The \emph{reaping number} $\mathfrak r(\kappa)$ is the smallest size of an unsplit family at~$\kappa$.
  \item A \emph{strongly unsplit family} at $\kappa$ is a set $F\subseteq\ub_\kappa$ such that for every $a\subseteq\kappa$, there is $b\in F$ for which either $a\cap b=\emptyset$ or $b\setminus a=\emptyset$.
  \item We let $\mathfrak R(\kappa)$ denote the smallest size of a strongly unsplit family at $\kappa$.
  \item A \emph{cone covering family} at $\kappa$ is a set $\mathcal{F} \subseteq \heart_{\ub_\kappa}$ such that \[\bigcup_{f \in\mathcal{F}} [f] = 2^\kappa.\]
  \item $R(\kappa)$ is the smallest size of a cone covering family at $\kappa$.
\item We let $R^*(\kappa)$ denote the smallest size of a family $\mathcal{F}\subseteq\heart_{\ub_\kappa}$ such that there exists a comeager set $X\subseteq 2^\kappa$ with $X\subseteq \bigcup_{f \in\mathcal{F}} [f]$.

\end{itemize}
\end{definition}

The next lemma collects some easy basic facts about these cardinal invariants, together with the result that $\mathfrak R(\kappa)=R(\kappa)$.

\begin{lemma}\label{lemma:reapingbasics}
  $\kappa^+\le\mathfrak r(\kappa)\le\mathfrak R(\kappa)=R(\kappa)\le 2^{<\kappa}\cdot\mathfrak r(\kappa)\le 2^\kappa$, and $\kappa^+\leq R^*(\kappa)\le R(\kappa)$.
\end{lemma}
\begin{proof}
  It is well-known (and very easy to check) that $\kappa^+\le\mathfrak r(\kappa)\le 2^\kappa$. Clearly, every strongly unsplit family at $\kappa$ is an unsplit family at $\kappa$, and this directly yields that $\mathfrak r(\kappa)\le\mathfrak R(\kappa)$. If $F$ is an unsplit family at $\kappa$, then \[F'=\{a\subseteq\kappa\mid\exists b\in F\ a\subseteq b\,\land\,|b\setminus a|<\kappa\}\] is clearly a strongly unsplit family at $\kappa$, of size $2^{<\kappa}\cdot|F|$, yielding that $\mathfrak R(\kappa)\le 2^{<\kappa}\cdot\mathfrak r(\kappa)$. That $R^*(\kappa)\le R(\kappa)$ is immediate from the definitions. To see that $\kappa^+\le R^*(\kappa)$ note that $[f]$ is nowhere dense for every $f\in \heart_{\ub_\kappa}$ (see Proposition~\ref{meagerbutnotImeager}), hence the union of $\kappa$-many such cones is meager. Therefore it cannot cover any comeager set, because $2^\kappa$ is not meager by the Baire Category Theorem for the bounded topology.
 It remains to show that $\mathfrak R(\kappa)=R(\kappa)$.

 $R(\kappa)\leq\mathfrak R(\kappa)$:
 Let $F$ be a strongly unsplit family at $\kappa$.
 For any set $A$ and $i \in 2$, let $c^A_i$ denote the function with domain $A$ and constant value $i$.
 We show that $\{c^b_i\with b\in F,\,i\in 2\}$ is a cone covering family at $\kappa$. Since $F$ is a strongly unsplit family, for every $x\in 2^\kappa$, there is $b\in F$ and $i\in 2$ such that
$x^{-1}[\{i\}]\cap b=\emptyset$.
 Therefore, $x\in [c^b_{1-i}]$.

$\mathfrak R(\kappa)\leq R(\kappa)$: Let $\mathcal F$ be a cone covering family at $\kappa$. Let \[F := \{ f^{-1}[\{i\}]\mid f \in\mathcal F,\,i \in 2 \}\cap\ub_\kappa.\]
We show that $F$ is a strongly unsplit family at $\kappa$. Let $a\subseteq \kappa$. Since $\mathcal F$ is a cone covering family, there is $f\in\mathcal F$ with $\chi_a\in [f]$. Thus, $f^{-1}[\{1\}]\subseteq a$, and hence $f^{-1}[\{1\}]\setminus a=\emptyset$.
On the other hand, $f^{-1}[\{0\}]\subseteq \kappa\setminus a$, hence $f^{-1}[\{0\}]\cap a=\emptyset$.
Since $\dom(f)\in\ub_\kappa$, either $f^{-1}[\{1\}]$ or $f^{-1}[\{0\}]$ has to be an unbounded subset of~$\kappa$, and hence an element of $F$.
\end{proof}

This shows that in particular if $2^{<\kappa}\le\mathfrak r(\kappa)$, then $\mathfrak r(\kappa)=\mathfrak R(\kappa)=R(\kappa)$. Our next result will show that if $\kappa$ is simple or $\kappa=\omega$,
$R^*(\kappa)$ is equal to the other invariants as well. We first need an auxiliary result:

\begin{lemma}\label{lemma:reaping}
  Let $\kappa$ be simple or $\kappa=\omega$,
and let $X$ be a comeager subset of $2^\kappa$. Then, there is a tree $T\subseteq 2^{<\kappa}$ of height $\kappa$ such that $[T]\subseteq X$, and such that the following two properties hold:
  \begin{enumerate}
\item[(a)]
$T$ has uniform splitting (i.e., on each level, either all nodes are splitting nodes, or none of them is) and the set of splitting levels of $T$ form a club\footnote{For $\kappa=\omega$, a club subset is just an unbounded subset.} subset $C$
of $\kappa$, such that $0\in C$.
\item[] For each $i < \kappa$, let $\prevspl{i}{C}$ be the largest $\delta \leq i$ with $\delta \in C$ (such a $\delta$ exists since $C$ is closed and contains $0$).
\item[(b)]
 For any two $x, y \in [T]$, and any $i < \kappa$, $x(i) = y(i)$ if and only if $x(\prevspl{i}{C}) = y(\prevspl{i}{C})$.
So we can define a function $\marlene\colon \kappa \times 2 \rightarrow 2$, such that for any $x \in [T]$ and any $i < \kappa$, we have
\begin{equation}\label{eqn:main_beth}
x(\prevspl{i}{C}) = \marlene(i,x(i)).
\end{equation}
\end{enumerate}
\end{lemma}
\begin{proof}
  In case $\diamondsuit_\kappa$ holds,
fix a $\diamondsuit_\kappa$-sequence $\langle A_\alpha\mid\alpha<\kappa\rangle$.

\medskip

Since $X$ is a comeager subset of $2^\kappa$, there exists a sequence $\langle D_i\mid i<\kappa\rangle$ of open dense subsets of $2^\kappa$ such that $X\supseteq\bigcap_{i<\kappa}D_i$. Since the intersection of less than $\kappa$-many open dense subsets of $2^\kappa$ is again open dense, we can assume that the sequence of $D_i$'s is $\subseteq$-decreasing.

  We construct $T$ by induction in $\kappa$-many steps, by constructing a sequence $\vec{S}=\langle S_i\mid i<\kappa\rangle$ of unboundedly many levels of $T$; we let $\delta_i$ be the
  length
  of the elements of $S_i$ (all sequences in $S_i$ will have the same
length)
  for all $i<\kappa$, and we let our club $C=\{\delta_i\mid i<\kappa\}$.
Let $S_0=\{\emptyset\}$.
Given $S_i$, consisting of binary sequences of equal length, we will define $S_{i+1}$ by distinguishing two cases.

Let us fix the following notation: for $u \in 2^{<\kappa}$,
let $1-u$ denote the sequence with all the bits of $u$ flipped, i.e.,
$1 - u \in 2^{<\kappa}$ is such that
$\dom(1-u) = \dom(u)$ and $(1-u)(\beta) = 1 -u(\beta)$ for each $\beta \in \dom(u)$.

\medskip

\textbf{Case 1: $\kappa$ is inaccessible or $\kappa=\omega$ --}
Let $\lambda := |S_i|$, and let $\{ t_l \with l<\lambda \}$ be an enumeration of~$S_i$. Note that $\lambda < \kappa$.
We will define an increasing sequence $\{ u^0_l \mid l < \lambda \}$ in $2^{<\kappa}$ by induction on $l < \lambda$, and afterwards, again by induction on $l < \lambda$, we will define another increasing sequence $\{ u^1_l \mid l < \lambda \}$ in $2^{<\kappa}$.

Let $u^0_0 := \emptyset$. For $l < \lambda$, let
$u^0_{l+1} \supseteq u^0_l$ be such that
\[
[t_l^\conc 0^\conc u^0_{l+1}] \subseteq D_i.
\]
At limits $l \leq \lambda$, let $u^0_l := \bigcup_{j < l} u^0_j$.
Let $u^1_0:=1- u^0_\lambda$.
For $l < \lambda$, let
$u^1_{l+1} \supseteq u^1_l$ be such that
\[
[t_l^\conc 1^\conc u^1_{l+1}] \subseteq D_i.
\]
At limits $l \leq \lambda$, let $u^1_l := \bigcup_{j < l} u^1_j$.

Let $u^1 := u^1_\lambda$, and let
$u^0:= 1-u^1$.
Finally, let
\[
S_{i+1} :=
\{ t_l^\conc 0^\conc u^0 \mid l < \lambda \} \cup \{ t_l^\conc 1^\conc u^1 \mid l < \lambda \}.
\]
This finishes the construction of $S_{i+1}$ in case $\kappa$ is inaccessible.
Note that every node in~$S_i$ is a splitting node of~$T$,
and
the elements of $S_{i+1}$ are sequences of equal length less than $\kappa$ (which we call $\delta_{i+1}$), using that $\lambda<\kappa$.
Finally, it is straightforward to check that
$[s] \subseteq D_i$ for each
$s \in S_{i+1}$.

\medskip

\textbf{Case 2: $\diamondsuit_\kappa$ holds --}
If $A_{\delta_i}\not\in S_i$, we simply let
$$
S_{i+1} = \{ t^\conc 0 \with t\in S_i\} \cup  \{ t^\conc 1 \with t\in S_i\}.
$$
Otherwise, we proceed as follows: Let
$u_0\in 2^{<\kappa}$ be such that $[A_{\delta_i}^\conc 0^\conc u_0]\subseteq D_{i}$, and let
$u_1 := 1-u_0$.
Let $v_1\in 2^{<\kappa}$ be such that $[A_{\delta_i}^\conc 1^\conc u_1^\conc v_1]\subseteq D_i$, and let
$v_0 := 1- v_1$.
Finally, let
\[S_{i+1}= \{t^\conc 0^\conc u_0^\conc v_0 \with t\in S_i\}\cup\{ t^\conc 1^\conc u_1^\conc v_1 \with t\in S_i\}.\]
This finishes the construction of $S_{i+1}$ in case $\diamondsuit_\kappa$ holds.
Note that, again, every node in~$S_i$ is a splitting node of~$T$,
and
the elements of $S_{i+1}$ are sequences of equal length less than $\kappa$ (which we call $\delta_{i+1}$).

\medskip

In both cases, for limit ordinals $i<\kappa$, we take unions, i.e., we let \[S_i= \{ t \in 2^{<\kappa} \with \exists\langle t_j\mid j<i\rangle\ t=\bigcup \limits_{j<i} t_j  \,\land\, t_j\in S_j\text{ for } j<i\}.\]
Note that elements of $S_i$ are again sequences of length less than $\kappa$ by the regularity of $\kappa$.
Finally, let $T=\{ t \with \exists i<\kappa\; \exists s\in S_i\; t\subseteq s\}$ be the tree induced by the $S_i$'s.
It remains to check that $T$ is as desired. It is straightforward to check from our construction that $T$ satisfies Properties (a) and (b) in either of our two cases. We have to check that moreover, our construction ensures that $[T]\subseteq\bigcap_{i<\kappa}D_i$ (and hence $[T]\subseteq X$).

\medskip

\textbf{Case 1: $\kappa$ is inaccessible or $\kappa=\omega$ --} Given $x\in[T]$ and $i<\kappa$,
note that $x\restr \delta_{i+1} \in S_{i+1}$, so, as discussed above, $[x\restr \delta_{i+1}] \subseteq D_i$, and hence $x \in D_i$, as desired.

\medskip

\textbf{Case 2: $\diamondsuit_\kappa$ holds --} Given $x\in[T]$ and $i<\kappa$, let $j\ge i$ be such that $A_{\delta_j}=x\restr\delta_j$, which is possible because $\{\delta_j \with i\le j<\kappa\}$ is a club and the $\diamondsuit_\kappa$-sequence guesses correctly on a stationary set.
Let $\xi := x(\delta_j)$. Note that,
by construction,
$x \in [A_{\delta_j}^\conc \xi^\conc u_\xi^\conc v_\xi] \subseteq D_j \subseteq D_i$, as desired.
\end{proof}

We are now ready to show that in many cases, all the cardinal invariants introduced in this section are actually equal.

\begin{theorem}\label{theorem:comeagerreaping}
 If $\kappa$ is simple or $\kappa=\omega$,
 then \[\mathfrak r(\kappa)=\mathfrak R(\kappa)=R(\kappa)=R^*(\kappa).\]
\end{theorem}
\begin{proof}
If $\kappa=\omega$, we clearly have $2^{<\kappa}=\kappa$; if $\kappa$ is simple, the same holds by Observation~\ref{obs:simple_implies}.
So,
  by Lemma \ref{lemma:reapingbasics}, we only need to show that $R^*(\kappa)\ge R(\kappa)$. Let $X$ be a comeager subset of $2^\kappa$ and let $\mathcal{F}\subseteq \heart_{\ub_\kappa}$ be such that $X\subseteq \bigcup_{f\in\mathcal F}[f]$. To finish the proof, it is enough to find such a family of the same size which covers the whole space~$2^\kappa$.

Let $T$, $C$, and $p$ be as provided by\footnote{Note that our assumption on $\kappa$ is only needed in order to be able to invoke Lemma \ref{lemma:reaping} at this point.} Lemma \ref{lemma:reaping} with respect to $X$, and let $\{ \delta_i \mid i<\kappa\}$ be an increasing enumeration of~$C$.
Since $[T]\subseteq X$ also $[T]\subseteq \bigcup_{f\in\mathcal F}[f]$.

 By passing to a suitable function $f'\subseteq f$ (with $f'\in \heart_{\ub_\kappa}$) for each $f\in\mathcal F$, we may assume that, for every $f\in\mathcal F$ and every $i<\kappa$, $\dom(f)$ and the ordinal interval $[\delta_i,\delta_{i+1})$ intersect in at most one element.\footnote{This uses that, clearly, if $f'\subseteq f$, then $[f']\supseteq[f]$.}

Now, define $\Omega,\Omega'\colon
\mathcal{F}\rightarrow\heart_{\ub_\kappa}$ as follows: for each $f \in\mathcal{F}$,
let
\begin{equation}\label{eqn:Omega_def}
\Omega(f) := \{ (\prevspl{i}{C}, \marlene(i, f(i))) \with i \in \mathrm{dom}(f) \}.
\end{equation}
Let $\Omega'(f)$  be
such that $i\in\mathrm{dom}(\Omega'(f))$ if and only if $\delta_i \in \mathrm{dom}(\Omega(f))$,
and, for such~$i$, let $\Omega'(f)(i) := \Omega(f)(\delta_i)$.
Clearly, for each $f\in\mathcal F$, $\Omega'(f)\in\heart_{\ub_\kappa}$,
and we want to finish our argument by showing that
\[
2^\kappa \subseteq \bigcup \limits_{f\in \mathcal{F}}[\Omega'(f)].
\]

To see this, let $\psi\colon 2^{<\kappa} \rightarrow T$ be the embedding for which for $s \in 2^{<\kappa}$, we have $\psi(s)(\delta_i) = s(i)$.\footnote{This is the unique isomorphism between $2^{<\kappa}$ and the set $\mathrm{split}(T)$ of splitting nodes of $T$ that preserves lexicographical order. } Let $\psi$ also denote its canonical extension to $2^\kappa$, i.e., for $y\in 2^\kappa$, let $\psi(y) = \bigcup_{i < \kappa} \psi(y \upharpoonright i)$.
Now let $y \in 2^\kappa$. Let $x := \psi(y)$. Since $x \in [T]$, we can fix $f \in \mathcal{F}$ such that $x \in [f]$, i.e., $x \supseteq f$.
By~\eqref{eqn:main_beth} from Lemma~\ref{lemma:reaping} and~\eqref{eqn:Omega_def}, it follows that $x \supseteq \Omega(f)$ as well.
Consequently, $y \supseteq \Omega'(f)$, i.e., $y \in [\Omega'(f)]$, thus finishing the argument.
\end{proof}

\begin{question}
  Is it consistent that $R^*(\kappa)<\mathfrak r(\kappa)$ for any regular uncountable cardinal~$\kappa$?
\end{question}

\section{Meager sets in ideal topologies}\label{section:meager}

In this section, we continue our investigation of the notion of meagerness in ideal topologies. One
point of the results of this section is to highlight some aspects of the complex relationship between the bounded topology and generalized ideal topologies on $2^\kappa$, and in particular the Edinburgh topology. Our first simple proposition shows that meagerness does not imply $\II$-meagerness.

\begin{proposition}\label{meagerbutnotImeager}
 If $f\in \heart_{\ub_\kappa}$, then $[f]$ is meager (in fact, closed nowhere dense).
 Thus, there is a meager subset of $2^\kappa$ which is not $\II$-meager whenever $\II\supsetneq\bd_\kappa$.
\end{proposition}
\begin{proof}
  $[f]$ is closed, for \[2^\kappa\setminus [f]=\bigcup\limits_{\alpha\in\dom(f), f(\alpha)\neq i } [\{(\alpha,i)\}]\] is open.

  Let $s\in\heart_{\bd_\kappa}$. Since $\dom(f)$ is unbounded in $\kappa$, we may pick some $\alpha\in\dom(f)\setminus\dom(s)$. Let $t=s\cup\{(\alpha,1-f(\alpha))\}\in\heart_{\bd_\kappa}$. Then, $[t]\cap [f]=\emptyset$, hence $[f]$ is nowhere dense.

Finally, if $\II\supsetneq\bd_\kappa$, there is $f \in \heart_\II \cap \heart_{\ub_\kappa}$; then $[f]$ is meager, but
$[f]$ is not $\II$-meager by Baire Category for the $\II$-topology (see Proposition~\ref{the:BCT}).
\end{proof}

We now show that there cannot be a small basis for the ideal~$\mathcal{M}_\II$ of $\II$-meager sets (provided that $\II$ is not the bounded ideal). The proof
uses a similar strategy as corresponding
proofs in the context of tree forcings on~$\omega$ which have large antichains
(see~\cite{cofbkw}).
As usual, let $\cof(\mathcal{M}_\II)$
denote the smallest size of a basis for the ideal~$\mathcal{M}_\II$.

\begin{proposition}
Let $\II\supsetneq \bd_\kappa$. Then $\cof(\mathcal{M}_\II)>2^\kappa$.
\end{proposition}
\begin{proof}
Let $\{ f_i \mid i<2^\kappa\}$ with $f_i\in \heart_\II$ be as in Observation~\ref{observation:manyopensets}(1), i.e., $\{ [f_i] \mid i<2^\kappa\}$ is a partition of $2^\kappa$. Fix $\{ X_i \mid i<2^\kappa\}$ with $X_i$
being
$\II$-meager. We will show that there exists an $\II$-nowhere dense (and hence $\II$-meager) set $Y$, which is not contained in any $X_i$. This shows that $\cof(\mathcal{M}_\II)>2^\kappa$.

Since $[f_i]$ is not $\II$-meager by Proposition~\ref{the:BCT}, we can pick $y_i\in [f_i]\setminus X_i$ for every $i<2^\kappa$. Let $Y:=\{ y_i \mid i<2^\kappa\}$. Clearly $Y\nsubseteq X_i$ for every $i<2^\kappa$. It remains to show that $Y$ is $\II$-nowhere dense.

 Let $f\in \heart_\II$. Clearly, there is an $i$ such that $[f]\cap [f_i]\neq \emptyset$ and hence $g:= f\cup f_i$ is in $\heart_\II$ and $[g]\subseteq [f_i]$. Since $[f_i]$ is disjoint from $[f_j]$ for all $j$ with $j\neq i$, there is at most one element in $Y\cap [g]$. Therefore, we can extend $g$ to $h\in \heart_\II$ such that $[h]\cap Y=\emptyset$, as desired.
\end{proof}

We now want to look at the question whether $\II$-meagerness could possibly imply meagerness, which we can answer negatively in many cases.

\begin{proposition}\label{proposition:smallinowheredense}
  If $\II$ is tall,\footnote{\label{footnote:weakly_tall}In fact, we only need that every $\II$-cone contains $2^\kappa$-many disjoint $\II$-cones, which is guaranteed by the following property of~$\II$ (which is actually weaker than tallness):
for each $Z \in \II$, there exists an unbounded set~$Z'$ in~$\II$ with $Z' \subseteq \kappa \setminus Z$.} then every set of size $<2^\kappa$ is $\II$-nowhere dense.
\end{proposition}
\begin{proof}
Assume $|X|<2^\kappa$, and let $f\in\heart_\II$. Using that $\II$ is tall, let $S\subseteq \kappa\setminus \dom(f)$ with $S\in\II\cap\ub_\kappa$. There are $2^\kappa$-many distinct extensions of $f$ to $\dom(f)\cup S$, and the cones of these are disjoint. Since $|X|<2^\kappa$, there exists $g\supseteq f$ in $\heart_\II$ with $[g]\cap X=\emptyset$. This shows that $X$ is $\II$-nowhere dense, as desired.
\end{proof}

Let $\mathcal{M_\kappa} = \mathcal{M_{\bd_\kappa}}$ denote the ideal of meager sets (in the bounded topology on~$2^\kappa$), and let
$\non(\mathcal{M_\kappa})$ denote its uniformity, i.e., the smallest size of a set which is not meager.
The following is an immediate consequence of the above proposition:

\begin{corollary}
If $\II$ is tall and $\non(\mathcal{M_\kappa})<2^\kappa$, then there is an $\II$-nowhere dense set that is not meager.
\end{corollary}

As a side note, let us show the following:

\begin{proposition}
There is a closed\footnote{This implies that the set is $\II$-closed for every ideal $\II$.} set $X$ of size $2^\kappa$ which is $\II$-nowhere dense for every ideal~$\II$. In fact, $X=[T]$ for a perfect subtree $T$ of $2^{<\kappa}$.
\end{proposition}
\begin{proof}
Let $X := \{ x \in 2^{\kappa} \mid x(2i)=x(2i+1) \textrm{ for each } i < \kappa \}$.

Let $\II$ be an ideal. To show that $X$ is $\II$-nowhere dense, let $f \in \heart_\II$.
Since $\kappa\notin \II$, $\dom(f) \neq \kappa$, so there exists an $i < \kappa$ such that either $2i \notin \dom(f)$ or $2i+1 \notin \dom(f)$.
Since $\bd_\kappa\subseteq \II$, we can extend $f$ to a function $g \in \heart_\II$ with $g \supseteq f$ such that $g(2i) \neq g(2i+1)$. Consequently, $[g] \subseteq [f]$ and $[g] \cap X = \emptyset$, as desired.
\end{proof}

The following yields another situation in which we obtain $\II$-meager sets that are not meager, and should also be contrasted with Corollary \ref{corollary:imeagerbairemeager}.

\begin{theorem} \label{theorem:indfromreaping}
 If $\kappa$ is\footnote{\label{footnote:GCH}What this proof actually needs is $R^*(\kappa)=2^\kappa$ and $2^{(2^{<\kappa})}=2^\kappa$.
So the assumptions of the theorem on $\kappa$ can also be replaced by GCH at $\kappa$ and $2^{<\kappa}=\kappa$ (which does not imply the assumptions of the theorem).
} simple, $\reap = 2^\kappa$, and
$\II$ is tall,\footnote{We need tallness only in order to be able to apply Proposition~\ref{proposition:smallinowheredense}. Therefore,
the weaker property from Footnote~\ref{footnote:weakly_tall} is in fact sufficient.}
then there exists an $\II$-nowhere dense set $X$ of size $2^\kappa$ which does not have the Baire property. In particular, $X$ is not meager.
\end{theorem}
\begin{proof}
	Let $\lambda=|2^\kappa|$. We construct $X$ in $\lambda$-many steps.
Fix an enumeration $\langle D_i\with i<\lambda\rangle$ of all $\kappa$-intersections of open dense sets (this is possible since  $2^{<\kappa}=\kappa$ by Observation~\ref{obs:simple_implies}), an enumeration $\langle[s_i]\with i<\kappa\rangle$ of the basic open sets (of the bounded topology on $2^\kappa$), and an enumeration $\langle f_i\with i<\lambda\rangle$ of $\heart_{\II\cap\ub_\kappa}$.
Let $\varphi\colon\lambda\rightarrow\lambda\times\kappa$ be a bijection.

\begin{enumerate}
\item Let $X_0=\emptyset$, $Y_0=\emptyset$.
\item Let $\varphi(i)=:(j,k)$. Let $X_{i+1}=X_i\cup \{x_{i}\}$, $Y_{i+1}=Y_i\cup \{y_i\}$ with
\begin{itemize}
  \item $x_i\in D_j\setminus Y_i$, $x_i\notin[f_l]$ for $l<i$, and
  \item $y_i\in [s_k]\cap D_j \setminus X_{i+1}$.
\end{itemize}
Such $x_i$ exists since $D_j$ is comeager, $Y_i$ has size less than $\lambda$, and $\reap=\lambda=R^*(\kappa)$ by assumption and by Theorem \ref{theorem:comeagerreaping}, hence\footnote{Note that $\bigcup\{[y] \mid y\in Y_i\}=\bigcup\{\{y\} \mid y\in Y_i\}=Y_i$.} $\{[f_l]\mid l<i\}\cup \{[y] \mid y\in Y_i\}$ cannot cover
$D_j$.
Such $y_i$ exists, because $D_j$ is comeager, thus $[s_k]\cap D_j$ is comeager in $[s_k]$, and therefore\footnote{It is well-known that comeager sets have full size, see
 Corollary~\ref{corollary:mycielski}.}
 $|[s_k]\cap D_j|=\lambda$ while $|X_{i+1}|<\lambda$.
\item For limit ordinals $i\le\lambda$, let $X_i=\bigcup\limits_{j<i} X_j$ and $Y_i=\bigcup\limits_{j<i} Y_j$.
\end{enumerate}
Let $X=X_\lambda$ and $Y=Y_\lambda$.

$X$ is $\II$-nowhere dense: Let $f\in\heart_\II$.
Since $\II$ contains unbounded sets, we can pick $i < \lambda$ with $f \subseteq f_i$.
By construction, $|X\cap[f_i]|<2^\kappa$,
but every set of size $<2^\kappa$ is $\II$-nowhere dense by Proposition~\ref{proposition:smallinowheredense}.
So there exists $g \in \heart_\II$ with $[g] \subseteq [f_i] \subseteq [f]$ such that $X\cap [g]=\emptyset$, as desired.

$X$ does not have the Baire property:
Let $U$ be open.
If $U=\emptyset$,  $X\Delta U = X$, and $X$ is not meager, because
$X\cap D \neq \emptyset$
for every comeager set $D$.
If $U\neq \emptyset$, $X\Delta U\supseteq U\setminus X \supseteq Y \cap [s_k]$ for some $k<\kappa$. But $Y\cap [s_k]\cap D\neq \emptyset$ for every comeager set $D$. Thus, $X\Delta U$ is not meager.
\end{proof}

\begin{question}\label{question:invariants}
  Is there always an $\II$-meager set that is not meager, at least if $\kappa$ is simple? By our above results, this clearly relates to the question whether it is consistent that $\kappa$ is simple, and
  $\reap<\non(\mathcal{M}_\kappa)=2^\kappa$,
  which also appears to be open.
\end{question}

Certain consequences of $\reap<\non(\mathcal{M}_\kappa)=2^\kappa$ have appeared as open questions in the literature. For instance, it holds that $\mathfrak{b}_\kappa\leq\reap$ for all regular $\kappa$
(see \cite[Lemma~7 and the remark afterwards]{rs}). Brendle, Brooke-Taylor, Friedman and Montoya asked (\cite[Question 20 and Question 24]{bbm}) whether it is consistent that
$\mathfrak{b}_\kappa<\non(\mathcal{M}_\kappa)$ holds for some uncountable successor cardinal $\kappa$ with $2^{<\kappa}=\kappa$ or for some strongly inaccessible cardinal $\kappa$.

\section{The Baire property in the nonstationary topology}\label{section:baire}

In this section, we provide two results on the connections between the Baire property in the bounded and in the nonstationary topology.

\begin{proposition}\label{proposition:baire}
Assume that $\II=\NS_\kappa$.
There is a subset of $2^\kappa$ with the Baire property, but not the $\II$-Baire property.
\end{proposition}
\begin{proof}
Take any\footnote{In fact, this proof works for every ideal~$\II$ with the following two properties:
first, there is an $f\in \heart_{\ub_\kappa} \cap \heart_{\II}$ such that
$[f]$ is $\II$-homeomorphic to $2^\kappa$, and second,
the assumption of Corollary~\ref{club filter is no Edinburgh Borel}(3) is satisfied, i.e.,
$\II$ is not stationarily tall.}
 $f\in \heart_{\ub_\kappa} \cap \heart_{\II}$.
Since $[f]$ is nowhere dense (see Proposition~\ref{meagerbutnotImeager}), every subset of $[f]$ is meager and thus has the Baire property.

Since $\NS_\kappa$ does not contain any stationary set, it is not stationarily tall, hence by Corollary~\ref{club filter is no Edinburgh Borel}(3), $2^\kappa$
has a subset without the $\II$-Baire property.
Since $2^\kappa$ and $[f]$ are $\II$-homeomorphic by
Lemma~\ref{lem:every_cone_homeomorphic},
there exists a subset of $[f]$ without the $\II$-Baire property (which has the Baire property by the above).
\end{proof}

Note that on the other hand, by Theorem \ref{theorem:indfromreaping}, there is an Edinburgh meager set without the Baire property, assuming that $\kappa$ is simple and $\reap=2^\kappa$ (or $2^{<\kappa}=\kappa$ and $2^\kappa=\kappa^+$, see Footnote~\ref{footnote:GCH}).

\medskip

We will need the following, which may also be of independent interest:

\begin{observation}\label{observation:simpleissimple}
Let $\kappa$ be simple, let $\PP$ denote $\kappa$-Silver forcing, and let $\II=\NS_\kappa$. Then, $X\subseteq 2^\kappa$ satisfies the $\II$-Baire property if and only if for every $f\in\PP$, there is $g\le f$ in $\PP$ such that either $[g]\subseteq X$ or $[g]\cap X=\emptyset$.\footnote{It suffices to assume that each $\II$-meager set is $\II$-nowhere dense (which is in particular the case if $\kappa$ is simple and $\II = \NS_\kappa$).}
\end{observation}
\begin{proof} This is straightforward to check, using that under our assumptions, every $\II$-meager set is $\II$-nowhere dense (see Theorem~\ref{theorem:cones2}).

Alternatively, it follows from our observation in Section~\ref{section:connection} that $\kappa$-Silver forcing is topological and generates the Edinburgh topology on $2^\kappa$, together with  the combination of \cite[Lemma 3.8, 1 and 2]{fkk} and of Theorem \ref{theorem:axioma}.
\end{proof}

The following is shown for inaccessible $\kappa$ in \cite[Lemma 4.9, 6]{fkk},
yet, making use of Theorem \ref{theorem:cones}, it can also be shown to hold under the assumption of $\diamondsuit_\kappa$. It also shows that the statement of \cite[Corollary 3.14]{fkk} for $\kappa$-Silver forcing can be generalized to include the case that $\diamondsuit_\kappa$ holds: indeed, note that for regular and uncountable cardinals $\kappa$, adding $\kappa^+$-many Cohen subsets of $\kappa$ forces that every $\boldsymbol{\Delta}^1_1$ subset of $2^\kappa$ has the Baire property (see for example \cite[Theorem~3.13(1)]{fkk}),
 and also forces $\diamondsuit_\kappa$ by an easy folklore standard argument.

\begin{theorem}\label{Delta11Baire}
  If $\kappa$ is simple and  every $\boldsymbol{\Delta}^1_1$ subset of $2^\kappa$ has the Baire property, then every $\boldsymbol{\Delta}^1_1$ subset of $2^\kappa$ has the $\II$-Baire property for $\II=\NS_\kappa$.\footnote{In the language of \cite{fkk}, this means that whenever $\kappa$ is simple, then $\boldsymbol{\Delta}^1_1(\mathbb C_\kappa)\to\boldsymbol{\Delta}^1_1(\mathbb V_\kappa)$, where $\mathbb C_\kappa$ denotes $\kappa$-Cohen forcing, and $\mathbb V_\kappa$ denotes $\kappa$-Silver forcing.}
\end{theorem}
\begin{proof}
  Let $\PP$ denote $\kappa$-Silver forcing, and let $\Gamma$ denote the collection of $\boldsymbol{\Delta}^1_1$ subsets of $2^\kappa$. All we will actually need in the argument, as in \cite[Lemma 4.9]{fkk}, is that $\Gamma$ is closed under continuous preimages.

Let $X\in\Gamma$, and let $f\in\PP$.
By Observation \ref{observation:simpleissimple} it is enough to show that there exists a non-empty $\II$-open subset of~$[f]$ which is either contained in $X$ or disjoint from~$X$.

By
Lemma~\ref{lem:every_cone_homeomorphic},
there exists $\varphi\colon 2^\kappa\rightarrow [f]$ which is an homeomorphism with respect to the bounded topology and with respect to the
nonstationary topology
(and the respective induced topologies on $[f]$).
Let $X'=\varphi^{-1}[X]$, which is again in $\Gamma$ because $\varphi$ is continuous, and hence has the Baire property by our assumption. This means that either $X'$ is meager, or it is comeager in some basic open set $[s]$ of the bounded topology on $2^\kappa$. If $X'$ is meager, then $2^\kappa\setminus X'$ is comeager, so Theorem~\ref{theorem:cones} yields an $\II$-cone $[g]$ that is disjoint from~$X'$.
If it is comeager in $[s]$, then there is a comeager set $D$ such that $D\cap [s]\subseteq X'$, so Theorem \ref{theorem:cones}
yields an $\II$-cone $[g] \subseteq D \cap[s] \subseteq X'$.

 But then, since $\varphi$ is an $\II$-homeomorphism, $\varphi[[g]]\subseteq[f]$ is an $\II$-open set that is either disjoint from or contained in $X$.
\end{proof}

\bibliography{edinburgh}
\bibliographystyle{alpha}

\end{document}